\documentclass[12pt]{amsart}
\usepackage[all, knot]{xy}
\usepackage{epsfig}
\usepackage{xspace}
\usepackage{layout}
\usepackage{latexsym}
\usepackage{amsmath}
\usepackage{amsthm}
\usepackage{amssymb}
\usepackage{amsfonts}
\usepackage{amsxtra}     
\usepackage{url}

\usepackage{verbatim}
\usepackage{longtable}

\newcommand{\vcomp}[2]{\genfrac{[}{]}{0pt}{}{#1}{#2}}
\newcommand{\bbA}{\mathbb{A}}

\newcommand{\bbC}{\mathbb{C}}
\newcommand{\bbD}{\mathbb{D}}
\newcommand{\bbE}{\mathbb{E}}

\newcommand{\bbH}{\mathbb{H}}
\newcommand{\bbM}{\mathbb{M}}

\newcommand{\bbQ}{\mathbb{Q}}

\newcommand{\bbV}{\mathbb{V}}
\newcommand{\bbX}{\mathbb{X}}
\newcommand{\Sp}{\mathbb{S}\text{\rm pan}}
\newcommand{\bEnd}{\bbE\text{\rm nd}}
\newcommand{\bMnd}{\bbM\text{\rm nd}}

\newcommand{\Qbar}{\overline{\bbQ}}

\newcommand{\bfA}{\mathbf{A}}

\newcommand{\bfC}{\mathbf{C}}

\newcommand{\bfH}{\mathbf{H}}

\newcommand{\bfV}{\mathbf{V}}
\newcommand{\bfX}{\mathbf{X}}

\newcommand{\bfK}{\mathbf{K}}

\newcommand{\Obj}{\operatorname{Obj}}
\newcommand{\Mor}{\operatorname{Mor}}
\newcommand{\Hor}{\operatorname{Hor}}
\newcommand{\Ver}{\operatorname{Ver}}

\newcommand{\Sq}{\operatorname{Sq}}

\newcommand{\Cat}{\text{\sf Cat}}

\newcommand{\End}{\operatorname{End}}
\newcommand{\Mnd}{\operatorname{Mnd}}

\newcommand{\Und}{\operatorname{Und}}
\newcommand{\Inc}{\operatorname{Inc}}

\newcommand{\Alg}{\text{-}\mathrm{Alg}}

\newcommand{\co}{\colon\thinspace}

\newcommand{\tb}[1]{\phantom{\sum^\Sigma_\Sigma} #1 \phantom{\sum^\Sigma_\Sigma}}
\newcommand{\lr}[1]{\hspace{.5mm}#1\hspace{.5mm}}

\newcommand{\pr}{\text{\rm pr}}

\newcommand{\id}{\operatorname{id}}

\begin{document}

\newtheorem{thm}{Theorem}[section]
\newtheorem{conj}[thm]{Conjecture}
\newtheorem{lem}[thm]{Lemma}
\newtheorem{cor}[thm]{Corollary}
\newtheorem{prop}[thm]{Proposition}
\newtheorem{rem}[thm]{Remark}

\theoremstyle{definition}
\newtheorem{defn}[thm]{Definition}
\newtheorem{examp}[thm]{Example}
\newtheorem{notation}[thm]{Notation}
\newtheorem{rmk}[thm]{Remark}

\theoremstyle{remark}

\makeatletter
\renewcommand{\maketag@@@}[1]{\hbox{\m@th\normalsize\normalfont#1}}%
\makeatother

\renewcommand{\labelenumi}{(\roman{enumi})}
\renewcommand{\labelenumii}{(\alph{enumii})}

\renewcommand{\theenumi}{(\roman{enumi})}
\renewcommand{\theenumii}{(\alph{enumii})}

\def\square{\hfill ${\vcenter{\vbox{\hrule height.4pt \hbox{\vrule width.4pt
height7pt \kern7pt \vrule width.4pt} \hrule height.4pt}}}$}

\newenvironment{pf}{{\it Proof:}\quad}{\square \vskip 12pt}
\date{\today}

\title[Double Adjunctions and Free Monads]{Double Adjunctions and Free Monads}
\author[T. M. Fiore]{Thomas M. Fiore}
\address{Thomas M. Fiore, Department of Mathematics and Statistics, University of
Michigan-Dearborn, 4901 Evergreen Road, Dearborn, MI 48128, USA}
\email{tmfiore@umd.umich.edu}

\author[N. Gambino]{Nicola Gambino}
\address{Nicola Gambino, Dipartimento di Matematica e Informatica, Universit\`a degli Studi di
Palermo, via Archirafi 34, 90123 Palermo, Italy}
\email{ngambino@math.unipa.it}

\author[J. Kock]{Joachim Kock}
\address{Joachim Kock, Departament de Matem\`atiques,
Universitat Aut\`onoma de Barcelona, 08193 Bellaterra (Barcelona),
Spain} \email{kock@mat.uab.cat}




\begin{abstract}
{\bf R\'esum\'e.}
   Nous caract\'erisons les adjonctions doubles en termes de
   pr\'efaisceaux et carr\'es universels, puis appliquons ces
   caract\'erisations aux monades libres et aux objets
   d'Eilenberg--Moore dans les cat\'egories doubles.  Nous am\'eliorons
   notre resultat paru dans \cite{FioreGambinoKock:DoubleMonadsI} comme suit~:
   si une cat\'egorie double munie d'un co-pliage admet la construction
   des monades libres dans sa 2-cat\'egorie horizontale, alors elle
   admet aussi la construction des monades libres en tant que
   cat\'egorie double.  Nous y d\'emontrons aussi qu'une cat\'egorie
   double admet les objets d'Eilenberg--Moore si et seulement si un
   certain pr\'efaisceau param\'etris\'e est repr\'esentable.  Pour ce faire,
   nous d\'eveloppons une notion de pr\'efaisceaux
   param\'etris\'es sur les cat\'egories doubles et d\'emontrons un
   lemme de Yoneda pour icelles.
\\
{\bf Abstract.}  We characterize double adjunctions in terms of presheaves and
universal squares, and then apply these characterizations to free monads and
Eilenberg--Moore objects in double categories. We improve upon our
earlier result in \cite{FioreGambinoKock:DoubleMonadsI} to conclude:
    if a double category with cofolding admits the construction of
    free monads in its horizontal 2-category, then it also admits
    the construction of free monads as a double category.
We also prove that a double
category admits Eilenberg--Moore objects if and only if a certain
parameterized presheaf is representable. Along the way, we develop
parameterized presheaves on double categories and prove a double-categorical
Yoneda Lemma.
\end{abstract}

%

\maketitle

\tableofcontents



\section{Introduction}

The notion of double category was introduced by Ehresmann~\cite{Ehresmann:CR-1963-01-28}
in 1963, as an instance of the concept of
internal category from \cite{ehresmann}, and was developed in the context of a
general theory of structure, as synthesized in his book {\em Cat\'egories et
structures} \cite{ehresmann2} (published in 1965), which in many regards was
ahead of its time.
Meanwhile, B\'enabou in his thesis work (under Ehresmann's
supervision) emphasized the simpler notion
of $2$-category, discovered that
$\mathbf{Cat}$ itself is an example, and derived the notion from that of
enrichment (Cat\'egories relatives)~\cite{Benabou:CR-1965-04-05}.
$2$-categories rather than double categories became the standard setting for
$2$-dimensional structures in category theory, not only because of a more
generous supply of examples, but also because $2$-categories behave and feel a
lot more like $1$-categories, whereas double categories present certain strange
phenomena.  For example not every compatible arrangement of squares in a double
category is composable, see~Dawson--Par\'e~\cite{dawsonpare1993}.  The past decade, however,
with the proliferation of higher-categorical viewpoints and methods, has seen a
certain renaissance of double categories, and double-categorical structures are
being discovered and studied more and more frequently in many different areas,
while also traditional $2$-categorical situations are being revisited in the new
light of double categories.

We became interested in double categories through work in conformal field
theory, topological quantum field theory, operad theory, and categorical logic.
In all these cases, the double-categorical structures come about in situations
where there are two natural kinds of morphisms, typically some complicated
morphisms (like spans of sets or bimodules) and some more elementary ones (like
functions between sets or ring homomorphisms), and the double-categorical
aspects concern the interplay between such different kinds of morphisms.  While
it often provides great conceptual insight to have everything encompassed in a
double category, one is often confronted with the lack of machinery for dealing
with double categories, and a need is being felt for a more systematic theory of
double categories.

This paper can be seen as a small step in that direction: although our work is
motivated by some concrete questions about monads, we develop further
the basics of adjunctions between double categories: we introduce
parametrized presheaves, prove a double Yoneda Lemma, characterize adjunctions
in several ways, and go on to study double categories with further structure ---
foldings or cofoldings --- for which we study the question of existence of free
monads and Eilenberg--Moore objects.  This was our original motivation, and in
that sense the present paper is a sequel to our previous paper
\cite{FioreGambinoKock:DoubleMonadsI} about {\em monads in double categories},
although logically it is rather a precursor: with the theory we develop here,
some of the results from \cite{FioreGambinoKock:DoubleMonadsI} can be
strengthened and simplified at the same time.

The notion of adjunction we consider is that of internal adjunction in $\mathbf{Cat}$.
There are two such notions: horizontal and vertical, depending on the
interpretation of double categories as internal categories.  A more general
notion of vertical double adjunction was studied by Grandis and
Par\'e~\cite{Grandis-Pare:adjoints}; we comment on the relationship in
Section~\ref{sec:Double_Adjunctions}.  Although horizontal and vertical
adjunctions are abstractly equivalent notions, under transposition of double
categories, often the double categories have extra structure which breaks the
symmetry and makes the two notions different.  In this paper we need both
notions.

In some regards, double adjunctions express universality in the ways one expects
based on experience with 1-categories, as we prove in
Theorem~\ref{thm:double_adjunction_descriptions}: a horizontal double adjunction may be
given by double functors $F$ and $G$ with horizontal natural transformations
$\eta$ and $\varepsilon$ satisfying the two triangle identities, or by double
functors $F$ and $G$ with a universal horizontal natural transformation ($\eta$
or $\varepsilon$), or by a single double functor $F$ or $G$ equipped with
appropriate universal squares compatible with vertical composition, or by a
bijection between sets of squares compatible with vertical composition.

This article primarily deals with {\it strict} double categories and strict double adjunctions,
and the unmodified term ``double category'' always means ``{\it strict} double category''. However, we do develop a result about horizontal adjunctions between normal, vertically weak double categories in Theorem~\ref{thm:double_adjunction_def_iff_v_pseudo}. Its transpose applies to the free--forgetful adjunction between endomorphisms and monads in the normal, horizontally weak double category $\Sp$ of horizontal spans, see the final paragraphs of Section~\ref{sec:Notation} for more on ``pseudo'' versus ``strict'' and the example in Section~\ref{sec:DoubleAdjunctionBetweenEndosInSpanAndMonadsInSpan}.

Although double adjunctions express universality in some of the ways one expects,
the characterizations of adjointness in 1-category theory in terms of
representability do not carry over to double category theory in a
straightforward way, and instead require a new notion of {\it presheaf on a
double category}.  Namely, to prove that an ordinary 1-functor $F\colon \bfA
\rightarrow \bfX$ admits a right adjoint, it is sufficient to show that the
presheaf $\bfA(F-,A)$ is representable for each object $A$ separately.  But to
establish that a {\it double} functor $F$ admits a horizontal right {\it double} adjoint,
two new requirements arise: first, we must consider how the analogous presheaves
{\it vertically combine}, and second, we must consider the representability of
all the analogous presheaves {\it simultaneously rather than separately}.  The
first requirement forces presheaves on double categories to be {\it vertically
lax} and to {\it take values in the normal, vertically weak double category} $\Sp^t$ {\it of vertical
spans}, as opposed to the 1-category $\mathbf{Set}$.  We prove a Yoneda Lemma
for such $\Sp^t$-valued presheaves in
Proposition~\ref{prop:Double_Yoneda_Lemma}.  The second requirement leads us to
consider {\it parameterized} presheaves on double categories.  With these
notions we establish the double-categorical analogue of the representability
characterization of adjunctions in
Theorem~\ref{thm:representability_characterization_of_dbladjs}, namely a double
functor admits a horizonal right adjoint if and only if a certain parameterized
$\Sp^t$-valued presheaf is representable.  Parameterized presheaves also play a
role in the proof of Theorem~\ref{thm:double_adjunction_descriptions}.

Yoneda theory for double categories has been studied also in a recent paper by Par\'e \cite{Pare:DoubleYoneda}.
He independently obtains our Examples~\ref{exa:homsets_are_a_parametrized_presheaf} and \ref{exa:representables} (his Section 2.1), Proposition~\ref{prop:Double_Yoneda_Lemma} on the Double Yoneda Lemma (his Theorem 2.3), and Theorem~\ref{thm:double_adjunction_descriptions}~\ref{laxapproach} (his Theorem 2.8).

Many double categories of interest have additional structure that
  allows one to reduce certain questions about the double category to
  questions about the horizontal $2$-category.  Two such structures
  are {\it folding} and {\it cofolding}, recalled in Definitions~\ref{def:folding}
  and \ref{def:cofolding}.  Double categories with both folding and
  cofolding are essentially the same as {\it framed bicategories} in
  the sense of Shulman~\cite{Shulman:Framed}.  In this article we work with
foldings and cofoldings separately because some examples, including our
motivating examples, admit one or the other but not both.

As an example of the principle of reduction to the horizontal 2-category in the
presence of a folding or cofolding,
Proposition~\ref{prop:folding_double_adjunction_iff_horizontal_double_adjunction}
states that two double functors $F$ and $G$ compatible with foldings (or
cofoldings) are horizontal double adjoints if and only if their underlying
horizontal 2-functors are 2-adjoints.

It is a much more subtle question to deduce a {\em vertical} double adjunction
from a $2$-adjunction in the horizontal $2$-category.  We discuss the special
cases of quintet double categories in the second half of
Section~\ref{sec:Double_Adjunctions_with_Foldings_Cofoldings}.  Surprisingly
such a deduction
is possible in the case of our main result, Theorem~\ref{thm:existence_of_free_monads},
which concerns monads in double
categories and the free-monad adjunction, as we proceed to explain.  In
our earlier paper \cite{FioreGambinoKock:DoubleMonadsI} we showed how to
associate to a double category $\bbD$ a double category $\bEnd(\bbD)$ of
endomorphisms in $\bbD$ and a double category $\bMnd(\bbD)$ of monads in $\bbD$.
The double categories $\bEnd(\bbD)$ and $\bMnd(\bbD)$ are extensions of Street's
2-categories of endomorphisms and monads in \cite{Street:formal-monads} in the
sense that if $\mathbf{K}$ is a 2-category and $\bbH(\mathbf{K})$ is
$\mathbf{K}$ viewed as a vertically trivial double category, then the horizontal
2-categories of $\bEnd(\bbH(\mathbf{K}))$ and $\bMnd(\bbH(\mathbf{K}))$ are
Street's 2-categories $\End(\mathbf{K})$ and $\Mnd(\mathbf{K})$.  In
\cite[Theorem~3.7]{FioreGambinoKock:DoubleMonadsI} we established a fairly
technical criterion which allows one to conclude the existence of free monads in
a double-categorical sense from the existence of free monads in the underlying
horizontal 2-category.  The basic assumptions were that the double category is a
framed bicategory and the appropriate substructures admit 1-categorical
equalizers and coproducts.  In the present paper we clarify and generalize this,
using the theory of double adjunctions and cofoldings.

A double category $\bbD$ is said to {\it admit the construction of free monads}
if the forgetful double functor $\bMnd(\bbD) \rightarrow \bEnd(\bbD)$ admits a {\it vertical}
left double adjoint such that the underlying vertical morphism of each unit
component is the identity.  This is somewhat more stringent than our earlier
definition in \cite{FioreGambinoKock:DoubleMonadsI}, where we required only a
vertical left double adjoint.  Our main application,
Theorem~\ref{thm:existence_of_free_monads}, states
that a double category with cofolding admits the construction of free monads if
its horizontal 2-category admits the construction of free monads.  This improves
\cite[Theorem~3.7]{FioreGambinoKock:DoubleMonadsI}, since it removes most of the
technical hypotheses and also strengthens the conclusion.  A main step is
Proposition~\ref{prop:cofolding_on_D_induces_cofolding_on_Mnd(D)_and_on_End(D)},
which states that a cofolding on a double category $\bbD$ induces cofoldings on
$\bEnd(\bbD)$ and $\bMnd(\bbD)$.  The corresponding statement for foldings does
not seem to be true.

To illustrate the theory, we consider in detail the example of the normal, horizontally weak double
category $\Sp$ of horizontal spans.  In $\Sp$, the endomorphisms are directed graphs and monads are categories.  The double-categorical free--forgetful
adjunction between $\bEnd(\Sp)$ and $\bMnd(\Sp)$ extends the classical
construction of the free category on a graph.


Returning to general double categories without cofolding, we now describe our
second main application.  Theorem~\ref{thm:EilenbergMooreExistence} states that
a double category $\bbD$ admits Eilenberg--Moore objects if and only if the
parameterized presheaf is representable which assigns to a monad $(X,S)$ and an
object $I$ in $\bbD$ the set $S$-$\text{Alg}_I$ of $S$-algebra structures on $I$.
The proof is quite short, since most of the work was done in the earlier
sections.

\smallskip

\noindent {\bf Outline of the paper.} Section~\ref{sec:Notation} presents our
notational conventions.  In Section~\ref{sec:presheaves} we introduce
parameterized presheaves on double categories and their representability, and
prove the Double Yoneda Lemma.  In Sections~\ref{sec:Universal_Squares} and
\ref{sec:Double_Adjunctions} we introduce universal
squares, and prove the various characterizations of horizontal double
adjunctions.
Section~\ref{sec:Double_Adjunctions_with_Foldings_Cofoldings} is concerned
with the case of horizontal double adjunctions compatible with foldings and cofoldings.
In Section~\ref{sec:EndosAndMonadsInADoubleCategory} we prove that $\bEnd(\bbD)$
and $\bMnd(\bbD)$ admit cofoldings when $\bbD$ does.
Section~\ref{sec:DoubleAdjunctionBetweenEndosInSpanAndMonadsInSpan} works out
the vertical double adjunction between $\bEnd(\Sp)$ and $\bMnd(\Sp)$ explicitly.
Sections~\ref{sec:Construction_of_Free_Monads} and
\ref{sec:Characterization_of_Existence_of_Eilenberg-Moore_Objects} are
applications of the results on double adjunctions to the construction of free
monads in double categories with cofolding and to a characterization of the
existence of Eilenberg--Moore objects in a general double category.

\section{Notational Conventions}\label{sec:Notation}

We begin by fixing some notation concerning double categories.

A {\it double category} is a categorical structure consisting of objects,
horizontal morphisms, vertical morphisms, squares, the relevant domain and
codomain functions, compositions, and units, subject to a few axioms~\cite
{Ehresmann:CR-1963-01-28}.  Succinctly, a double category is an internal
category in $\mathbf{Cat}$ \cite{ehresmann}, and in particular involves
a diagram of categories
and functors
$$
\bbD_1 \times_{\bbD_0} \bbD_1  \stackrel m \longrightarrow \bbD_1
\stackrel u\longleftarrow \bbD_0 .
$$
Here $\bbD_0$ is the category of objects and vertical arrows of $\bbD$,
and $\bbD_1$ is the category of horizontal arrows and squares, and
$m$ and $u$ express horizontal composition and identity cells.

The notion was introduced by
C.~Ehresmann in the mid sixties and investigated by A.~Ehresmann and C.~Ehresmann
in the 60's and 70's; among those pioneering works on the subject, the
most relevant for the present paper are \cite{Ehresmann:CR-1963-01-28,
ehresmann, EhresmannQuintets, ehresmann2}.  We refer to
Bastiani--Ehresmann~\cite{ehresmannone}, Brown--Mosa~\cite{BrownMosa},
Fiore--Paoli--Pronk~\cite{FiorePaoliPronk}, and Grandis--Par\'e~\cite{Grandis-Pare:limits}
for more modern treatments, each
starting with a short introduction to double categories.  The homotopy theory of
double categories has been investigated by
Fiore--Paoli~\cite{FiorePaoli} and Fiore--Paoli--Pronk~\cite{FiorePaoliPronk}.

We indicate double categories with blackboard letters, such as $\bbC$, $\bbD$,
and $\bbE$, and denote horizontal respectively vertical composition of squares
by
\begin{equation} \label{bracket_notation_for_square_compositions}
\left[ \alpha \; \; \beta \right] \;\; \text{ and } \;\;
\begin{array}{c}\begin{bmatrix} \alpha \\ \gamma \end{bmatrix}\end{array},
\end{equation}
when they are defined. The double category axiom called {\it interchange law} then states the equality
\begin{equation} \label{equ:interchange_law}
\begin{bmatrix}
\begin{bmatrix}
\alpha & \beta
\end{bmatrix} \vspace{1mm} \\
\begin{bmatrix}
\gamma & \delta
\end{bmatrix}
\end{bmatrix}=\begin{bmatrix}
\begin{bmatrix} \alpha \\ \gamma \end{bmatrix} &
\begin{bmatrix} \beta \\ \delta \end{bmatrix}
\end{bmatrix}.
\end{equation}
We simply denote this composite by
\begin{equation} \label{equ:interchange2}
\begin{bmatrix} \alpha & \beta \\ \gamma & \delta
\end{bmatrix}.
\end{equation}
The notation in \eqref{bracket_notation_for_square_compositions} similarly applies to horizontal and vertical morphisms, for instance, $\left[f \;\; g \right]$ and $\vcomp{j}{k}$ denote the composites $g \circ f$ and $k \circ j$ in the usual orthography. The horizontal and vertical identity morphisms on an object $C$ in $\bbC$ are denoted
$1^h_C$ and $1^v_C$ respectively. The horizontal identity square for a vertical morphism $j$ is denoted by $i^h_j$, while the vertical identity square for a horizontal morphism $f$ is indicated with $i^h_f$.

If $\bbD$ is a double category, then $\Hor \bbD, \Ver \bbD,$ and $\Sq \bbD$,
signify the collections of horizontal morphisms, vertical morphisms, and squares
in $\bbD$.  To specify the set of horizontal respectively vertical morphisms
from an object $D_1$ to an object $D_2$, we write $\Hor \bbD(D_1,D_2)$ and $\Ver
\bbD(D_1,D_2)$.  Similarly, the notation $\Hor \bbD(f,g)$ indicates the function
$\Hor\bbD(D_1,D_2)\rightarrow \Hor\bbD(D_1',D_2')$ obtained by pre- and
postcomposition with the horizontal morphisms $f$ and $g$.  The function $\Ver
\bbD(j,k)$ is defined analogously.  To indicate the collection of squares with
fixed left vertical boundary $j$ and fixed right vertical boundary $k$, we write
\begin{equation} \label{hom_notation}
\bbD(j,k)=\left\{ \alpha \in \Sq \bbD \; \bigg| \; \alpha \text{ has
the
form } \begin{array}{c} \xymatrix{\ar[r] \ar[d]_j \ar@{}[dr]|{\alpha} & \ar[d]^k \\
\ar[r] &  } \end{array} \right\}.
\end{equation}
For example, for the vertical identities $1^v_{D_1}$ and $1^v_{D_2}$, the set
$\bbD(1^v_{D_1},1^v_{D_2})$ consists of the 2-cells between morphisms $D_1
\rightarrow D_2$ in the horizontal 2-category of $\bbD$.  In general, the
squares in $\bbD(j,k)$ may not compose vertically.  Also in analogy to the
hom-notation, the notation $\bbD(\alpha,\beta)$ means horizontal pre- and
postcomposition by squares $\alpha$ and $\beta$.

For any double category
$\bbD$, the {\it horizontal opposite} $\bbD^{\text{horop}}$ is formed by
switching horizontal domain and codomain for both horizontal morphisms and
squares in $\bbD$.  More precisely, the horizontal 1-category of
$\bbD^{\text{horop}}$ is equal to the opposite of the horizontal 1-category of
$\bbD$, the vertical 1-category of $\bbD^{\text{horop}}$ is the same as that of
$\bbD$, and the category $(\Ver \bbD^{\text{horop}}, \Sq \bbD^{\text{horop}})$
is equal to the opposite category of $(\Ver \bbD, \Sq \bbD)$.

The {\em transpose}
of a double category is obtained by switching the vertical and horizontal
directions.  The symmetric nature of the notion of double category means that
each double category has two different interpretations as an internal category;
these two interpretations are interchanged by transposition.
 We shall always stick with the ``horizontal'' interpretation outlined
initially.

{\em Double functors} are just internal functors, and the same notion results
from the two possible interpretations of double categories as internal
categories.  We shall also need {\em vertically lax double functors}: these
strictly preserve horizontal composition, but provide non-invertible
comparison $2$-cells for composition of vertical arrows.  We refer to
Grandis--Par\'e~\cite{Grandis-Pare:adjoints} for the details.  A {\em horizontal
natural transformation} is an internal natural transformation in $\mathbf{Cat}$
(for our preferred internal interpretation). In particular, a horizontal natural transformation $\theta\colon F \Rightarrow G$ for $F,G\colon \mathbb{D} \to \mathbb{E}$ assigns to each object $A$ of $\mathbb{D}$ a horizontal
morphism $\theta A \colon FA \to GA$, and assigns to each vertical morphism $j$ in $\bbD$ a square $\theta j$ bounded on the left and right by $Fj$ and $Gj$ respectively, such that
\begin{equation} \label{equ:horizontal_natural_transformation}
\theta 1^v_A=i^h_{1^v_A}\hspace{.5in}\theta \vcomp{j_1}{j_2}=\vcomp{\theta j_1}{\theta j_2} \hspace{.5in} \left[ F\alpha \;\; \theta k \right]=\left[\theta j \;\; G\alpha \right]
\end{equation}
for all objects $A$ of $\bbD$, composable vertical morphisms $j_1$ and $j_2$ of $\bbD$, and squares $\alpha$ in $\bbD(j,k)$.
A {\em vertical natural transformation} can be defined as an internal natural transformation for the
transposed internal interpretation, which is the same as the transpose of the horizontal notion above, but can also be described succinctly as
follows: a vertical natural transformation $\theta$ between two double functors
$F,G \co \bbA \to \bbX$ consists of two natural transformations $\theta_0 \co
F_0 \Rightarrow G_0$ and $\theta_1 \co F_1 \Rightarrow G_1$ compatible with
horizontal composition and identity cells.

Double categories, double functors and horizontal natural transformations
form a $2$-category $\mathbf{DblCat_h}$, and there is a canonical $2$-functor
$\bfH \co \mathbf{DblCat_h} \to \mathbf{2Cat}$ which to a double category
associates its horizontal $2$-category, i.e.~consisting of objects, horizontal
arrows and squares whose vertical sides are identities.  Similarly there is
a $2$-category $\mathbf{DblCat_v}$ of double categories, double functors, and
vertical natural transformations, and a canonical $2$-functor $\bfV\co
\mathbf{DblCat_v} \to \mathbf{2Cat}$ defined similarly as $\bfH$.

The double category $\bbV_1\bbD$ has vertical 1-category the
vertical 1-category of $\bbD$ and everything else trivial, that is, there are no
non-trivial squares and no non-trivial horizontal morphisms in $\bbV_1\bbD$.
The subscript 1 in $\mathbb{V}_1 \bbD$ reminds us that we retain only the
vertical 1-category part of $\bbD$, and also distinguishes $\bbV_1 \bbD$ for a
double category $\bbD$ from $\bbV\bfK$ for a 2-category $\bfK$, which we define
momentarily.

A 2-category $\bfK$ gives rise to various double categories.
The double category $\bbH\bfK$ has $\bfK$ as its horizontal
2-category and only trivial vertical morphisms.  Similarly, the double category
$\bbV\bfK$ has $\bfK$ as its vertical 2-category and only trivial horizontal
morphisms.  Double categories of quintets of a $2$-category will be introduced
in Examples~\ref{exa:quintets} and \ref{exa:twistedquintets}.

In this paper, the term ``double category'' always means ``{\it strict} double category.'' We predominantly
work with strict double categories, except for a few specified passages: in Section~\ref{sec:presheaves} the normal, vertically weak double category $\Sp^t$ is the codomain of presheaves on double categories, Theorem~\ref{thm:double_adjunction_def_iff_v_pseudo} concerns double adjunctions of strict double functors between horizontally weak double categories, and Section~\ref{sec:DoubleAdjunctionBetweenEndosInSpanAndMonadsInSpan} treats the main example of the free--forgetful double adjunction between the normal, vertically weak double categories $\bEnd(\Sp)$ and $\bMnd(\Sp)$.

To explain the meaning of this terminology, recall that a {\it pseudo double
category} is like a double category, except one of the two morphism
compositions (vertical or horizontal) is associative and unital up to coherent
invertible squares, rather than strictly,
cf.~Grandis--Par\'e~\cite{Grandis-Pare:limits}, see also Chamaillard~\cite{Chamaillard},
Fiore~\cite{Fiore:PseudoAlgebrasPseudoDoubleCategories}, Martins-Ferreira~\cite{Martins-Ferreira}. In this article we specify the weak direction in a given pseudo category by our usage of the terms {\it horizontally weak double category} and {\it vertically weak double category}. In either case, the interchange law in \eqref{equ:interchange_law} holds strictly.

All of the pseudo double categories we work with will also be {\it normal}, that is, the coherent unit squares are actually identity squares, so that the identity morphisms in the weak direction are strict identities. As mentioned in \cite[page 172]{Grandis-Pare:limits}, this is easily arranged for pseudo double categories in which the weakly associative composition is given by some kind of choice (e.g. choice of pullbacks in the case of $\Sp$ in Example~\ref{exa:span}).

Normality has useful consequences. For each vertical morphism $j$, the square $i^h_j$ is an identity for the horizontal composition of squares (in a general pseudo category, $i^h_j$ is merely a distinguished square compatible with vertical composition). This small detail is needed in the proof of Theorem~\ref{thm:double_adjunction_def_iff_v_pseudo}. Another consequence of normality is that $\bfV\bbD$ is a strict 2-category when $\bbD$ is a normal, horizontally weak double category. If $\bbD$ is horizontally weak and not normal, then $\bfV\bbD$ is neither a bicategory nor a 2-category (however the vertical composition of 2-cells in $\bfV\bbD$ can be redefined to make a 2-category). See pages 44-46 of \cite{Fiore:PseudoAlgebrasPseudoDoubleCategories}, especially Remark~6.2, for a discussion of these topics.

Note also that (strict) horizontal natural transformations make sense between double functors of normal, vertically weak double categories (see the requirements in \eqref{equ:horizontal_natural_transformation}).

\begin{examp}
  \label{exa:span}
The normal, horizontally weak double category $\Sp$ will play a special role in this paper.  Its
objects are sets, its horizontal morphisms are spans of sets, its vertical morphisms are functions, and its squares are morphisms of spans. The horizontal composition of morphisms is by pullback combined with function composition: for the composite of two nontrivial horizontal morphisms, we choose the usual model for a set-theoretic pullback, which is a subset of the Cartesian product, and then compose the projections with remaining maps in the spans. However, for the composite of a horizontal morphism $B \leftarrow A \rightarrow C$ with an identity, we {\it choose the pullback} to be simply $A$. This choice of pullback makes the horizontally weak double category $\Sp$ {\it normal}, that is, the horizontal identities are actually strict horizontal identities. Consequently, for any two vertical morphisms $j$ and $k$ in $\Sp$, the horizontal identity squares $i^h_j$ and $i^h_k$ actually satisfy $\left[i^h_j \; \; \alpha \right]=\alpha=\left[\alpha \; \; i^h_k \right]$.

The normal, vertically weak double category $\Sp^t$ is the transpose of $\Sp$. Note that $\Sp$ is horizontally weak while $\Sp^t$ is vertically weak.
\end{examp}

\section{Parameterized Presheaves and the Double Yoneda Lemma}
\label{sec:presheaves}

In this section we introduce and study parametrized
presheaves, and prove a Yoneda Lemma for double categories.  The Double
Yoneda Lemma in Proposition~\ref{prop:Double_Yoneda_Lemma} and the
characterization of horizontal left double adjoints in
Theorem~\ref{thm:representability_characterization_of_dbladjs} require
parameterized $\Sp^t$-valued presheaves, as explained in the Introduction.  The
covariant Double Yoneda Lemma for presheaves on a double category $\bbD$ says
that morphisms from the represented presheaf $\bbD(R,-)$ to a presheaf $K$ on
$\bbD^{\text{horop}}$ are in bijective correspondence with the set $K(R)$.

A presheaf on a double category assigns to objects sets, to horizontal morphisms
functions, to vertical morphisms spans of sets, and to squares morphisms of
spans.  Moreover, these image spans are equipped with a kind of composition
provided by the vertical laxness of the presheaf, see for example equation
\eqref{equ:kind_of_composition}.

\begin{defn} \label{def:presheaf}
  Let $\bbD$ be a double category.
  \begin{enumerate}
    \item
A {\it presheaf} on $\bbD$ is a
vertically lax double functor $\bbD^\text{horop} \rightarrow \Sp^t$.

\item
A {\it morphism of presheaves} is a horizontal natural
transformation of vertically lax double functors
$\bbD^\text{horop} \rightarrow \Sp^t.$
\end{enumerate}
\end{defn}

\begin{defn} \label{def:parameterized_presheaf}
Let $\bbD$ and $\bbE$ be double categories.
\begin{enumerate}
  \item
A {\it presheaf on
$\bbD$ parameterized by $\bbE$} is a vertically lax double functor
$\bbD^\text{horop}\times \bbE \rightarrow \Sp^t$. We
synonymously use the term {\it presheaf on $\bbD$ indexed by
$\bbE$}.

\item
A {\em morphism of presheaves on $\bbD$ parameterized by $\bbE$} is just a horizontal
natural transformation between them.
\end{enumerate}
\end{defn}

\begin{examp} \label{exa:homsets_are_a_parametrized_presheaf}
The most basic example is delivered by the hom-sets of a double category
$\bbD$. Namely, a presheaf on $\bbD$ indexed by $\bbD$ is defined on
objects and horizontal morphisms by
$$\xymatrix{\bbD(-,-)\co \bbD^\text{horop}\times \bbD \ar[r] & \Sp^t }$$
$$\xymatrix{(D_1,D_2) \ar@{|->}[r] & \Hor \bbD (D_1,D_2) }$$
$$\xymatrix{(f,g) \ar@{|->}[r] & \Hor \bbD (f,g) }.$$
On vertical morphisms $(j,k)$, it is the vertical span
$$\xymatrix{\Hor \bbD(s^v j,s^v k) \\ \bbD(j,k)  \ar[d]^{t^v} \ar[u]_{s^v} \\
\Hor \bbD(t^v j,t^v k),}$$
which we often denote simply by $\bbD(j,k)$. On squares
$(\alpha,\beta)$, the vertically lax double functor $\bbD(-,-)$ is
the morphism of vertical spans induced by
$\bbD(\alpha,\beta)(\gamma)=\left[\alpha \; \; \gamma \;\; \beta
\right]$ as well as $\Hor \bbD(s^v \alpha,s^v \beta)$ and $\Hor
\bbD(t^v \alpha,t^v \beta)$.

For the vertically lax double functor $\bbD(-,-)$, the composition
coherence square in $\Sp^t$
$$\begin{array}{c} \begin{bmatrix} \bbD(j,k) \\ \bbD(\ell,m)
\end{bmatrix}\end{array} \hspace{-2ex}
\xymatrix{ \ar[r] & \bbD(\vcomp{j}{\ell},\vcomp{k}{m})} $$ is simply
composition in $\bbD$. More precisely, on elements we have
\begin{equation} \label{equ:kind_of_composition}
\begin{array}{c} \xymatrix{ \ar[r] \ar[d]_j \ar@{}[dr]|{\xi_1} & \ar[d]^k \\
  \ar[d]_\ell \ar[r] \ar@{}[dr]|{\xi_2}
& \ar[d]^m \\ \ar[r] & } \end{array} \xymatrix{ \ar@{|->}[r] & }
\begin{array}{c} \xymatrix{ \ar[r] \ar[dd]_{\vcomp{j}{\ell}}
  \ar@{}[ddr]|{\vcomp{\xi_1}{\xi_2}}
& \ar[dd]^{\vcomp{k}{m}} \\ & \\  \ar[r] & } \end{array}.
\end{equation}
The unit coherence square in $\Sp^t$ of the vertically lax double functor
$\bbD(-,-)$ is simply the vertical identity square embedding
$$\xymatrix{1^v_{\bbD(D_1,D_2)} \ar[r]^{i^v} & \bbD(1^v_{D_1},1^v_{D_2})}$$
$$f \xymatrix{\ar@{|->}[r] & } \begin{array}{c} \xymatrix{D_1 \ar@{=}[d]
\ar[r]^f \ar@{}[dr]|{i^v_f} & D_2 \ar@{=}[d]
\\ D_1 \ar[r]_f & D_2}\end{array}.$$
The presheaf $\bbD(-,-)$ may also be considered as a presheaf on
$\bbD^{\text{horop}}$ indexed by $\bbD^{\text{horop}}$.
This completes the example $\bbD(-,-)$.
\end{examp}

\begin{examp} \label{exa:representables}
As a special case of Example~\ref{exa:homsets_are_a_parametrized_presheaf}, we
may fix the first variable to be an object $R$ in $\bbD$ and we obtain a
presheaf on $\bbD^\text{horop}$, namely
$$\xymatrix{\bbD(R,-)\co \bbD \ar[r] & \Sp^t}.$$
This presheaf is {\it represented} by the object $R$. We shall
discuss a notion of representability for parameterized presheaves in
Definition~\ref{defn:parametrized_presheaf_representability}, as
they will be a key ingredient in our characterizations of horizontal double
adjunctions in
Theorem~\ref{thm:double_adjunction_descriptions}~\ref{laxapproach}
and Theorem~\ref{thm:representability_characterization_of_dbladjs}.

We write out the features of
Example~\ref{exa:homsets_are_a_parametrized_presheaf} for this
special case, since we will need these represented presheaves in the
Double Yoneda Lemma. Like any double functor, this presheaf consists
of an object functor and a morphism functor
$$\xymatrix{\bbD (R,-)^\text{Obj}\co(\Obj \bbD_0, \Obj \bbD_1) \ar[r] &
(\text{Sets}, \text{Functions})}$$
$$\xymatrix{\bbD (R,-)^\text{Mor}\co(\Mor \bbD_0, \Mor \bbD_1) \ar[r] &
(\text{Spans}, \text{Morphisms of Spans})}.$$
The object functor is the usual represented presheaf on the
horizontal 1-category, namely
$$\bbD(R,D)^\text{Obj}:=\{f\co R \rightarrow D \; \vert \; f
\text{ horizontal morphism in } \bbD \}=\Hor \bbD(R,D)$$
$$\bbD(R,g)^\text{Obj}(f):=\left[f \; \; g \right].$$
The morphism functor, on the other hand, takes a vertical morphism
$j\co D \rightarrow D'$ in $\bbD$ to the (vertical) span
$\bbD(R,j)^\text{Mor}$ defined as
$$\xymatrix{\bbD(R,D)^\text{Obj} \\ \ar[u]_-{s^v} \bbD(1^v_R,j) \ar[d]^-{t^v} \\
\bbD(R,D')^\text{Obj},}$$
and on a square $\beta$ we have the morphism of spans
$\bbD(R,\beta)^\text{Mor}$ induced by
$\bbD(R,\beta)^\text{Mor}(\alpha)=\left[\alpha \; \; \beta \right].$

The composition coherence square in $\Sp^t$
$$\begin{array}{c} \begin{bmatrix} \bbD(R,j)^\text{Mor} \\ \bbD(R,k)^\text{Mor}
\end{bmatrix}\end{array} \hspace{-2ex}
\xymatrix{ \ar[r] & \bbD(R,\vcomp{j}{k})} $$ of the vertically lax
double functor $\bbD(R,-)$ is simply composition in $\bbD$. More
precisely, on elements we have
$$\begin{array}{c} \xymatrix{R \ar[r] \ar@{=}[d] \ar@{}[dr]|{\xi_1} & \ar[d]^j \\
R \ar@{=}[d] \ar[r] \ar@{}[dr]|{\xi_2}
& \ar[d]^k \\ R \ar[r] & } \end{array} \xymatrix{ \ar@{|->}[r] & }
\begin{array}{c} \xymatrix{R \ar[r] \ar@{=}[dd] \ar@{}[ddr]|{\vcomp{\xi_1}{\xi_2}}
& \ar[dd]^{\vcomp{j}{k}} \\ & \\ R \ar[r] & } \end{array}.$$

The unit coherence square in $\Sp^t$ of the vertically lax double
functor $\bbD(R,-)$ is simply the identity embedding
$$\xymatrix{1^v_{\bbD(R,D)^\text{Obj}} \ar[r]^{i^v} & \bbD(R,1^v_D)^\text{Mor}}$$
$$f \xymatrix{\ar@{|->}[r] & } \begin{array}{c} \xymatrix{R \ar@{=}[d] \ar[r]^f
\ar@{}[dr]|{i^v_f} & D \ar@{=}[d]
\\ R \ar[r]_f & D}\end{array}.$$
\end{examp}

\begin{examp}
If $\bfC$ is a 1-category, then a classical presheaf on $\bfC$ may
be considered a presheaf on $\bbH\bfC$ in the following way. A
classical presheaf on $\bfC$ is the same thing as a strictly unital
double functor $F\co (\bbH\bfC)^{\text{horop}} \rightarrow \Sp^t$
which has composition coherence morphism for $F(1_C^v) \circ
F(1_C^v) \rightarrow F(1_C^v)$ given by the projection of the
diagonal of $FC \times FC$ to $FC$. Any presheaf on $\bbH\bfC$
restricts to a classical presheaf on $\bfC$ by forgetting $F(1_C^v)$
for each $C$ and the composition and identity coherences.
\end{examp}

\begin{examp}
A presheaf on the (opposite of the) terminal double category 1 is
the same as a category, since a vertically lax double functor from 1
into $\Sp^t$ is the same as a (horizontal) monad in $\Sp$, which is
the same as a category. Note also that morphisms of such presheaves
are horizontal natural transformations of vertically lax double functors,
hence are the same as functors (see \cite{FioreGambinoKock:DoubleMonadsI}).
\end{examp}

\begin{examp}
Let $\bfC$ be a 1-category.  Then $\bfC(-,-)$ is a presheaf on $\bfC$ indexed by
$\Obj \bfC$.  This is a way to consider all the presheaves $\bfC(-,C)$
simultaneously.  Similarly, by parametrizing via the vertical 1-category of
$\bbD$, the indexed presheaf $\bbD(-,-)\co \bbD^\text{horop} \times \bbV_1\bbD
\rightarrow \Sp^t$ is a way of considering all presheaves $\bbD(-,R)$ {\it
simultaneously and how they combine vertically} (recall the notation
$\bbV_1\bbD$ from Section~\ref{sec:Notation}).  This point of view will become
important for our characterization of horizontal double adjunctions in
Theorems~\ref{thm:double_adjunction_descriptions}~and~\ref{thm:representability_characterization_of_dbladjs}.
\end{examp}

\begin{defn} \label{defn:parametrized_presheaf_representability}
A parameterized presheaf $F \co \bbD^{\text{horop}}\times \bbE \rightarrow
\Sp^t$ in the sense of Definition~\ref{def:parameterized_presheaf} is {\it
representable} if there exists a double functor $G\co \bbE \rightarrow \bbD$
such that $F$ is isomorphic to
$\bbD(-,G-)\co\bbD^\text{horop}\times \bbE \rightarrow \Sp^t$ as parameterized
presheaves.
\end{defn}

\begin{examp}
The presheaf $\bbD(-,R)\co \bbD^\text{horop} \rightarrow
\Sp^t$ is represented by the double functor $\ast
\rightarrow \bbD$ that is constant $R$. The indexed presheaf
$\bbD(-,-)\co \bbD^\text{horop} \times \bbV_1\bbD
\rightarrow \Sp^t$ is represented by the inclusion of the vertical
1-category of $\bbD$ into $\bbD$.
\end{examp}

We next prove the Double Yoneda Lemma. For simplicity, we do the
covariant version rather than the contravariant version.

\begin{prop}[Double Yoneda Lemma] \label{prop:Double_Yoneda_Lemma}
Let $\bbD$ be a small double category, $R$ an object of $\bbD$, $K\co \bbD
\rightarrow \Sp^t$ a vertically lax double functor, and $\text{\rm
HorNat}(\bbD(R,-),K)$ the set of horizontal natural transformations from
$\bbD(R,-)$ to $K$.  Then the map $$\xymatrix@R=.5pc{\theta_{R,K}\co\text{\rm
HorNat}(\bbD(R,-),K) \ar[r] & KR \\
\alpha \ar@{|->}[r] & \alpha_R(1_R^h)}$$ is a bijection. Further,
this bijection is a horizontal natural isomorphism of double
functors $N$ and $E$ $$\xymatrix{N, E\co \bbD \times \bbD\text{\rm
blCat}_{\text{\rm vert.lax}}(\bbD,\Sp^t) \ar[r] & \Sp^t }$$
$$N(R,K):=\text{\rm HorNat}(\bbD(R,-),K)$$
$$E(R,K):=K(R).$$
\end{prop}
\begin{proof}
This is an extension of the proof of Borceux~\cite[Theorem
1.3.3]{Borceux:HandbookI}.  We define $\theta_{R,K}(\alpha)=\alpha(1_R^h) \in
K(R)$ and for $a \in K(R)$ we define a horizontal natural transformation
$\tau(a)\co \bbD(R,-) \Rightarrow K$.  To each object $D\in \bbD$ we have the
horizontal morphism in $\Sp^t$
$$\xymatrix@R=.5pc{\tau(a)_D\co \bbD(R,D) \ar[r] & KD \\  f \ar@{|->}[r] & K(f)(a).}$$
and to each vertical morphism $j$ in $\bbD$ we have the square
$\tau(a)_j$ in $\Sp^t$
\begin{equation} \label{tausquare}
\begin{array}{c}
\begin{array}{c} \xymatrix@C=4pc{\bbD(R,D)^\text{Obj} \ar[r]^-{\tau(a)_D}  & K(D) \\
\bbD(1^v_R,j) \ar[r]|-{\lr{\tau(a)_j}} \ar[u] \ar[d] & \ar[u] \ar[d] K(j) \\
\bbD(R,D')^\text{Obj} \ar[r]_-{\tau(a)_{D'}}  &  K(D')}
\end{array} =
\begin{array}{c} \xymatrix@C=7pc{\bbD(R,D)^\text{Obj} \ar[r]^-{K(-)(a)}  & K(D) \\
\bbD(1^v_R,j) \ar[r]|-{\lr{K(-)(\delta^K_R(a))}} \ar[u] \ar[d] & \ar[u] \ar[d] K(j) \\
\bbD(R,D')^\text{Obj} \ar[r]_-{K(-)(a)}  &  K(D').}
\end{array}
\end{array}
\end{equation}
These squares commute, because for $\begin{array}{c} \xymatrix{R
\ar[r] \ar@{=}[d] \ar@{}[dr]|{\xi} & D \ar[d]^j \\ R \ar[r] & D'}
\end{array} \in \bbD(1^v_R, j)$ the squares
\begin{equation} \label{tausquarecommutes}
\begin{array}{c}
\xymatrix@C=7pc{K(R) \ar@{=}[d] \ar@{=}[r] & K(R) \ar[r]^{K(f)} & K(D) \\
K(R) \ar@{=}[d] \ar[r]|-{\lr{\delta^k_R}} & K(1^v_R) \ar[u] \ar[d]
\ar[r]|{\lr{K(\xi)}} & K(j) \ar[u] \ar[d] \\ K(R) \ar@{=}[r] & K(R)
\ar[r]_{K(f')} & K(D') }
\end{array}
\end{equation} commute. For example, the
top square in \eqref{tausquare} evaluated on $\xi$ is the same as the
top half of \eqref{tausquarecommutes} evaluated on $a$.

The naturality of $\tau(a)$, $\tau$, and $\theta$ is proved as in Borceux
\cite[Theorem 1.3.3]{Borceux:HandbookI}.
\end{proof}

\begin{cor}
For objects $R, S \in \bbD$, each horizontal natural transformation
$\bbD(R,-) \Rightarrow \bbD(S,-)$ has the form
$\bbD(h,-)$ for a unique horizontal arrow $h\co S \rightarrow  R$.
\end{cor}

\begin{rmk}
If $k$ is a vertical morphism in $\bbD$, then
$$\xymatrix{\bbD(k,-)\co (\Ver \bbD, \Sq \bbD) \ar[r] & (\text{Sets},
\text{functions})}$$
$$\hspace{7ex}\xymatrix@C=4pc{ \ell \ar@{|->}[r] & \bbD(k,\ell)}$$
is an ordinary presheaf on $(\Ver \bbD, \Sq \bbD)^{\text{op}}$.
\end{rmk}

\section{Universal Squares in a Double Category} \label{sec:Universal_Squares}

The components of the unit or counit of any 1-adjunction are universal
arrows. Conversely, a 1-adjunction can be described in terms of such
universal arrows. In this section we introduce universal squares in
a double category, with a view towards the analogous
characterizations of horizontal double adjunctions in
Theorem~\ref{thm:double_adjunction_descriptions}.

\begin{defn} \label{def:horizontally_universal_square}
If $S \co \bbD \rightarrow \bbC$ is a double functor, then a {\it (horizontally)
universal square from the vertical morphism $j$ to $S$} is a square $\mu$ in
$\bbC$ of the form
$$\xymatrix{C_1 \ar[d]_{j} \ar[r]^{u_1} \ar@{}[dr]|{\mu} & SR_1 \ar[d]^{Sk} \\
C_2 \ar[r]_{u_2} & SR_2}$$
such that the map
\begin{equation} \label{doubleuniversalarrowdiagram}
\begin{array}{c}
\xymatrix@R=.5pc{\bbD(k,\ell) \ar[r] & \bbC(j,S\ell) \\
\beta' \ar@{|->}[r] & \left[\mu \; \; S\beta' \right] }
\end{array}
\end{equation}
is a bijection for all vertical morphisms $\ell$. There is of course
a dual notion of {\it (horizontally) universal square from a double
functor $S$ to a vertical morphism $j$}.
\end{defn}

\begin{prop}
Suppose $S\co \bfK' \rightarrow \bfK$ is a
2-functor and $u\co C \rightarrow SR$ is a morphism in
$\bfK$. Then $\mu:=i_u^v$ is universal from $1^v_C$ to $\bbH
S$ if and only if the functor
\begin{equation*}
\xymatrix@C=4pc@R=.5pc{\bfK'(R,D) \ar[r]^{S(-) \circ u} &
\bfK(C,SD)\\ f' \ar@{|->}[r] & \left[u  \; \; Sf' \right]}
\end{equation*}
is an isomorphism of categories. In
other words, the square $i^v_u$ in $\bbH \bfK$ is universal
if and only if the morphism $u$ of $\bfK$ is 2-universal.
\end{prop}
\begin{proof}
In this situation the assignment $\beta' \mapsto \left[\mu \; \;
\bbH S \beta' \right]$ is a functor, namely whiskering with $u$.
Then the claim follows from the observation that the morphism part
of a functor is bijective if and only if the functor is an
isomorphism of categories.
\end{proof}

\begin{prop} \label{towardsrepresentability1}
The bijection in \eqref{doubleuniversalarrowdiagram} is a natural
transformation of functors
\begin{equation} \label{functors}
\xymatrix{\bbD(k,-) \ar@{=>}[r] & \bbC(j,S-)}.
\end{equation}

Conversely, given $k$ and $j$, any natural bijection of functors as
in \eqref{functors} arises in this way from a unique square $\mu \in
\bbC(j,Sk)$ which is universal from $j$ to $S$.
\end{prop}
\begin{proof}
The proof is very similar to that of Mac Lane~\cite[Proposition 1, page
59]{MacLane:working}.  The bijection is natural because $$\left[\mu \;\;
S\left[\beta' \;\; \gamma' \right] \right]=\left[\mu \; \; \left[S\beta' \;\;
S\gamma'
\right] \right].$$

For the converse, let $\phi \co \bbD(k,-) \Rightarrow
\bbC(j,S-)$ be a natural bijection, and define
$\mu:=\phi_k(i^h_k)$. The naturality diagram for $\phi$ and $\beta'$
yields $\left[\mu \;\; S\beta'\right]=\phi_{\ell}(\beta')$, which in
turn implies that \eqref{doubleuniversalarrowdiagram} is a bijection,
since $\phi_{\ell}$ is a bijection.
\end{proof}

For later use, we record the dual to Proposition~\ref{towardsrepresentability1}
using the inverse bijection.

\begin{prop} \label{towardsrepresentability2}
Universal squares in $\bbC(Sk,j)$ from $S\co \bbD \rightarrow \bbC$ to $j$ are
in bijective correspondence with natural bijections
$$\xymatrix{\bbC(S-,j) \ar@{=>}[r] & \bbD(-,k)}.$$
\end{prop}

\section{Double Adjunctions} \label{sec:Double_Adjunctions}
For any 2-category $\mathbf{K}$, there is a notion of adjunction
in $\mathbf{K}$ \cite{Kelly-Street:2cat}.
Namely, two 1-morphisms $f\co A \rightarrow B$ and $g\co B \rightarrow A$ in
$\mathbf{K}$ are adjoint if there exist 2-cells $\eta\co 1_A
\Rightarrow gf$ and $\varepsilon\co fg \Rightarrow 1_B$
satisfying the triangle identities.  From the 2-categories
$\mathbf{DblCat_h}$ and $\mathbf{DblCat_v}$ we thus get two notions of
adjunction between double categories.

\begin{defn} \label{global}
A {\em horizontal double adjunction} is an
adjunction in the 2-category $\mathbf{DblCat_h}$.
A {\em vertical double adjunction} is an adjunction in the
2-category $\mathbf{DblCat_v}$.
\end{defn}

The notions of horizontal and vertical adjunctions are of course transpose to
each other, so the result we list in this section for horizontal adjunctions are
also valid for vertical adjunctions.  However, as soon as the involved double
categories have further structure, like the foldings and cofoldings we consider
from Section~\ref{sec:Double_Adjunctions_with_Foldings_Cofoldings} and onwards,
the two notions behave differently.  In this paper we need both notions.

A more general notion of vertical adjunction was introduced and studied by
Grandis and Par\'e~\cite{Grandis-Pare:adjoints} (cf.~further comments below).
Vertical adjunctions were also studied by
Garner~\cite[Appendix~A]{Garner:DoubleClubs} and
Shulman~\cite[Section~8]{Shulman:Framed}.

   For the basic theory, which we treat in this section, we work only
   with
horizontal adjunctions.  The 2-category $\mathbf{DblCat_h}$ is the same as the
2-category $\Cat(\mathbf{Cat})$ of internal categories in $\mathbf{Cat}$,
internal functors, and internal natural transformations, which leads to various
characterizations of horizontal double adjunctions in terms of universal arrows
and bijections of hom-sets, along the lines of Mac Lane
\cite[Theorem 2, p.83]{MacLane:working}.  Our results in this vein in
Theorem~\ref{thm:double_adjunction_descriptions} can be deduced from more
general results of Grandis--Par\'e~\cite{Grandis-Pare:adjoints}, but we have
included the proofs since they are quite natural from the internal viewpoint
(which is not mentioned in \cite{Grandis-Pare:adjoints}).  The first
novelty comes when trying to characterize adjunctions in terms of presheaves:
here it turns out we need parametrized presheaves, which is the content of
Theorem~\ref{thm:representability_characterization_of_dbladjs}.

In Section~\ref{sec:DoubleAdjunctionBetweenEndosInSpanAndMonadsInSpan} we
present a completely worked example of a {\em vertical} double adjunction: the free
and forgetful double functors between endomorphisms and monads in $\Sp$.  This
is an extension of the classical adjunction between small directed graphs and
small categories.

\bigskip

Let $\bbA$ and $\bbX$ be double categories.  Since a horizontal double
adjunction
is precisely an internal adjunction, an explicit description is this:
a  horizontal double
adjunction from $\bbX$ to $\bbA$ consists of double functors
\begin{equation} \label{FG}
\xymatrix@C=4pc{\bbX \ar@/^1pc/[r]^{F} & \ar@/^1pc/[l]^{G} \bbA}
\end{equation}
and horizontal natural transformations
$$\xymatrix@1{\eta\co 1_{\bbX} \ar@{=>}[r] & GF}$$
$$\xymatrix@1{\varepsilon \co FG \ar@{=>}[r] & 1_{\bbA}}$$
such that the composites
$$\xymatrix@C=3pc{G \ar@{=>}[r]^-{\eta*i_G} & GFG \ar@{=>}[r]^-{i_G*\varepsilon} & G}$$
$$\xymatrix@C=3pc{F \ar@{=>}[r]^-{i_F*\eta} & FGF \ar@{=>}[r]^-{\varepsilon*i_F} & F}$$
are the respective identity horizontal natural transformations.
Here $F$ is the {\em horizontal left adjoint}, $G$ is the {\em horizontal right
adjoint}, and we write $F \dashv G$ to denote this horizontal adjunction.
In this section we consider only horizontal adjunctions, and suppress the
adjective ``horizontal'' for brevity.

\begin{thm}[Characterizations of horizontal double adjunctions] \label{thm:double_adjunction_descriptions}
A horizontal double adjunction $F \dashv G$ is completely determined by the
items in any one of the following lists.
\begin{enumerate}
\item \label{FGunituniversal}
Double functors $F$, $G$ as in \eqref{FG} and a horizontal natural
transformation $\eta\co 1_\bbX \Rightarrow GF$ such
that for each vertical morphism $j$ in $\bbX$, the square $\eta_j$
is universal from $j$ to $G$.
\item \label{Guniversalarrow}
A double functor $G$ as in \eqref{FG} and functors
$$\xymatrix{F_0\co
(\Obj \bbX, \Ver \bbX) \ar[r] & (\Obj \bbA, \Ver \bbA)}$$
$$\xymatrix{\eta \co (\Obj \bbX, \Ver \bbX) \ar[r] & (\Hor \bbX, \Sq \bbX)}$$
such that for each vertical morphism $j$ in $\bbX$ the square
$\eta_j$ is of the form
$$\xymatrix{X \ar[d]_j \ar[r]^-{\eta_X} \ar@{}[dr]|-{\eta_j} & GF_0X \ar[d]^{GF_0j} \\
Y \ar[r]_-{\eta_Y} & GF_0Y}$$ and is universal from $j$ to $G$. Then
the double functor $F$ is defined on vertical arrows by $F_0$ and on
squares $\chi$ by universality via the equation $\left[ \eta_{s\chi}
\; \; GF\chi \right] =\left[\chi \; \; \eta_{t\chi} \right]$.
\item \label{FGcounituniversal}
Double functors $F$, $G$ as in \eqref{FG} and a horizontal natural
transformation $\varepsilon \co FG \Rightarrow 1_{\bbA}$ such that for each
vertical morphism $k$ in $\bbA$, the square $\varepsilon_k$ is universal from
$F$ to $k$.
\item \label{Funiversalarrow}
A double functor $F$ as in \eqref{FG} and functors
$$\xymatrix{G_0\co
(\Obj \bbA, \Ver \bbA) \ar[r] & (\Obj \bbX, \Ver \bbX)}$$
$$\xymatrix{\varepsilon \co (\Obj \bbA, \Ver \bbA) \ar[r] & (\Hor \bbA, \Sq \bbA)}$$
such that for each vertical morphism $k$ in $\bbA$ the square
$\varepsilon_k$ is of the form
$$\xymatrix{FG_0A \ar[d]_{FG_0 k} \ar[r]^-{\varepsilon_{A}} \ar@{}[dr]|-{\varepsilon_k} & A \ar[d]^k \\
FG_0 B \ar[r]_-{\varepsilon_{B}} & B}$$ and is universal from $F$ to
$k$. Then the double functor $G$ is defined on vertical morphisms by
$G_0$ and on squares $\alpha$ by universality via the equation
$\left[FG \alpha \;\; \varepsilon_{t\alpha} \right] = \left[
\varepsilon_{s\alpha} \;\; \alpha \right]$.
\item \label{local}
Double functors $F$, $G$ as in \eqref{FG} and a bijection
$$\xymatrix@1{\varphi_{j,k}\co \bbA(Fj,k) \ar[r] & \bbX(j,Gk)}$$ natural
in the vertical morphisms $j$ and $k$ and compatible with vertical
composition.

Naturality here means natural as a functor
$$\xymatrix{(\Ver \bbX, \Sq \bbX)^{\text{\rm op}} \times (\Ver \bbA, \Sq \bbA) \ar[r] & \mathbf{Set}}.$$
That is, for squares $\sigma \in \bbX(j',j)$, $\alpha \in \bbA(Fj,k)$, $\tau \in \bbA(k,k')$ and squares $\sigma$, we have
$$\varphi\left( \left[F \sigma \;\; \alpha\right]\right)=\left[\sigma \;\; \varphi\left(\alpha \right)\right]\phantom{.}$$
$$\varphi\left( \left[ \alpha \;\; \tau \right] \right) = \left[ \varphi\left(\alpha \right) \;\; G\tau \right].$$

Compatibility with vertical composition means $$\varphi \left(
\begin{bmatrix}
\alpha \\ \beta \end{bmatrix} \right)=\begin{bmatrix}\varphi
(\alpha)
\\ \varphi(\beta)\end{bmatrix}.$$
\item \label{laxapproach}
Double functors $F$, $G$ as in \eqref{FG} and a horizontal natural
isomorphism between the vertically lax double functors
(parameterized presheaves)
$$\xymatrix@1{\bbA(F-,-)\co\bbX^\text{\rm horop} \times \bbA \ar[r] & \Sp^t}$$
$$\xymatrix@1{\bbX(-,G-)\co\bbX^\text{\rm horop} \times \bbA \ar[r] & \Sp^t}.$$
\end{enumerate}
\end{thm}

\begin{rmk}
  As mentioned, Grandis and Par\'e~\cite{Grandis-Pare:adjoints} have introduced
  a more general notion of double adjunction, which mixes colax and lax double
  functors, and due to this mixture, this notion is {\em not} an instance of an
  adjunction in a bicategory.  However, they observe that if at least one of the
  functors is pseudo (so that both functors can be considered colax or both
  lax), then the notion is the 2-categorical notion from the 2-category of
  double categories, either colax or lax double functors, and {\em vertical}
  natural transformations.  We just add to their observations that in the strict
  case we can transpose, and find that the strict version of their notion
  specializes to Definition~\ref{global} above.  Under these relationships,
  Theorem~\ref{thm:double_adjunction_descriptions} becomes essentially a special
  case of results of Grandis--Par\'e: characterization (v) is the transpose
  of the strict version of \cite[Theorem~3.4]{Grandis-Pare:adjoints}, and
  characterization (iv) is the transpose of the strict version of
  \cite[Theorem~3.6]{Grandis-Pare:adjoints}.  The other characterizations in
  Theorem~\ref{thm:double_adjunction_descriptions} are variations,
  but (vi) appears to be new.
\end{rmk}

\begin{proof}
We first prove Definition~\ref{global} is equivalent to \ref{local},
then we use this equivalence to prove the other equivalences (we provide much detail in the equivalence
Definition~\ref{global} $\Leftrightarrow$ \ref{local} because we will need these details for a pseudo version in
Theorem~\ref{thm:double_adjunction_def_iff_v_pseudo}).
In each equivalence, we omit the proof that the two procedures are
inverse to one another.\\
[2mm] \noindent Definition~\ref{global} $\Rightarrow$ \ref{local}.
Suppose $\langle F,G,\eta,\varepsilon\rangle$ is a double
adjunction. Then for any square $\gamma$ of the form
$$\xymatrix{\ar[r] \ar[d]_j \ar@{}[dr]|{\gamma} & \ar[d]^{\ell} \\ \ar[r] &  }$$ we have
$\left[ \eta_j \; \; GF\gamma \right]=\left[\gamma \; \;
\eta_{\ell}\right]$ by the horizontal naturality of $\eta$. We
define $\varphi_{j,k}$ and $\varphi_{j,k}^{-1}$ by
$$\varphi_{j,k}(\alpha):=\left[\eta_j \; \; G\alpha\right]$$
$$\varphi_{j,k}^{-1}(\beta):=\left[F \beta \; \; \varepsilon_k\right].$$
Then we have
$$\aligned \varphi\varphi^{-1} \beta &= \varphi\left[F\beta \; \; \varepsilon_k \right] \\
&= \left[\eta_j \; \; GF\beta \; \; G\varepsilon_k\right]
\\
&=\left[\beta \; \; \eta_{Gk} \; \; G \varepsilon_k \right] \text{ (by horizontal naturality)}\\
&=\beta \text{ (by triangle identity)}
\endaligned
$$
and similarly $\varphi^{-1}\varphi(\alpha)=\alpha$.

For the naturality of $\varphi_{j,k}$ in $k$, we have
$$\varphi\left( \left[ \alpha \;\; \tau \right] \right)\overset{\text{def}}{=} \left[\eta_j \;\; G[\alpha \;\; \tau]\right] = \left[\eta_j \;\; G\alpha \;\; G\tau\right] \overset{\text{def}}{=} \left[ \varphi\left(\alpha \right) \;\; G\tau \right].$$
Naturality of $\varphi_{j,k}$ in $j$ is similar, but additionally uses the naturality of $\eta$.

For the compatibility of $\varphi_{j,k}$ with vertical composition, we must use the interchange law from \eqref{equ:interchange_law} and the resulting convention \eqref{equ:interchange2}, as well as the compatibility of the horizontal natural transformation $\eta$ with vertical composition.
$$\begin{bmatrix}\varphi (\alpha) \\ \varphi(\beta)\end{bmatrix}=
\begin{bmatrix} \eta_j & G\alpha \\ \eta_{m} & G\beta \end{bmatrix}=
\begin{bmatrix} \eta_{\vcomp{j}{m}} & G\vcomp{\alpha}{\beta} \end{bmatrix}$$

We now have $\langle F,G,\varphi \rangle$
as in \ref{local}.\\
[2mm] \noindent \ref{local} $\Rightarrow$ Definition~\ref{global}.
From $\langle F,G,\varphi\rangle$ as in \ref{local}, we define
horizontal natural transformations by
$$\eta_j:=\varphi(i^h_{Fj})$$
$$\varepsilon_k:=\varphi^{-1}(i^h_{Gk}).$$

The assignment $\eta$ is natural because $i^h_{-}$ is a horizontal identity square
$$\left[\eta_j \;\; GF \gamma \right] \overset{\text{def}}{=} \left[ \varphi(i^h_{Fj}) \;\; GF \gamma \right]=\varphi\left[i^h_{Fj} \;\; F\gamma \right] = \varphi(\gamma)$$
$$\left[ \gamma \;\; \eta_{\ell} \right] \overset{\text{def}}{=} \left[\gamma \;\; \varphi(i^h_{F \ell}) \right] = \varphi \left[ F\gamma \;\; i^h_{F\ell} \right]= \varphi(\gamma).$$
For the compatibility of $\eta$ with vertical composition, we use the fact that $i^h_{-}$ is compatible with vertical composition
$$\eta_{\vcomp{j}{m}}\overset{\text{def}}{=} \varphi(i^h_{F\vcomp{j}{m}})=\varphi \begin{bmatrix}
i^h_{Fj} \\ i^h_{Fm} \end{bmatrix} = \begin{bmatrix}
\varphi(i^h_{Fj}) \\ \varphi(i^h_{Fm}) \end{bmatrix} \overset{\text{def}}{=}\begin{bmatrix}
\eta_j \\ \eta_m \end{bmatrix}.$$
The assignment $\varepsilon$ is similarly a horizontal natural transformation.

To verify that $\xymatrix@1@C=3pc{G \ar@{=>}[r]^-{\eta*i_G} & GFG \ar@{=>}[r]^-{i_G*\varepsilon} & G}$ is
the identity horizontal natural transformation on $G$ we have
$$\left[ \eta_{Gk} \;\; G(\varepsilon_k) \right]\overset{\text{def}}{=}\left[\varphi(i^h_{FGk}) \;\; G\varphi^{-1}(i^h_{Gk})\right]=\varphi\left[ i^h_{FGk} \;\; \varphi^{-1}(i^h_{Gk})\right]=i^h_{Gk}.$$
The proof of the other triangle identity is similar.

Finally, we now have $\langle F,G,\eta,\varepsilon\rangle$ as in Definition~\ref{global}.
We acknowledge the exposition of Mac Lane~\cite[pages 81--82]{MacLane:working} for this proof. \\
[2mm] \noindent \ref{FGunituniversal} $\Rightarrow$ \ref{local}.
Suppose we have $\langle F,G,\eta \rangle$ as in
\ref{FGunituniversal}. The universality of $\eta_j$ says that
\begin{equation} \label{etajuniversal}
\begin{array}{c}
\xymatrix@R=.5pc{\bbA(Fj,k) \ar[r] & \bbX(j,Gk) \\
\alpha \ar@{|->}[r] & \left[\eta_j \;\; G \alpha \right]}
\end{array}
\end{equation}
is a bijection. Clearly this bijection is natural in $j$ and $k$,
and compatible with vertical composition, so we obtain $\langle
F,G,\varphi \rangle$ as in
description \ref{local}. \\
[2mm] \noindent \ref{local} $\Rightarrow$ \ref{FGunituniversal}.
From the first part, we know that Definition~\ref{global} is
equivalent to \ref{local} and that
$\varphi_{j,k}(\alpha)=\left[\eta_j \; \; G\alpha\right]$. This
gives us $F$, $G$, and $\eta$. The universality of $\eta_j$ then
follows, because the map in \eqref{etajuniversal} is equal to
$\varphi_{j,k}$ and is therefore
bijective. \\
[2mm] \noindent \ref{FGunituniversal} $\Rightarrow$
\ref{Guniversalarrow}. The data in \ref{Guniversalarrow} are just a
restriction of the data in \ref{FGunituniversal}. \\
[2mm] \noindent \ref{Guniversalarrow} $\Rightarrow$
\ref{FGunituniversal}. The universality of $\eta_j$ guarantees that
for each square $\chi$ in $\bbX$ there is a unique square $F\chi$
such that $\left[ \eta_{s\chi} \; \; GF\chi \right] =\left[\chi \;
\; \eta_{t\chi} \right]$. This defines $F$ on squares $\chi$ in
$\bbX$, and we take $F$ to be $F_0$ on the vertical morphisms of
$\bbX$. Then $F$ is a double functor by the universality and the
hypothesis that $F_0$ and $\eta$ are functors. Finally, $\eta$ is
natural because of the defining equation $\left[ \eta_{s\chi} \; \;
GF\chi \right]
=\left[\chi \; \; \eta_{t\chi} \right]$. \\
[2mm] \noindent \ref{global} $\Leftrightarrow$
\ref{FGcounituniversal}. The proof of the equivalence
Definition~\ref{global} $\Leftrightarrow$ \ref{FGcounituniversal} is
dual to
the proof the equivalence Definition~\ref{global} $\Leftrightarrow$ \ref{FGunituniversal}. \\
[2mm] \noindent \ref{FGcounituniversal} $\Leftrightarrow$
\ref{Funiversalarrow}. The proof of the equivalence
\ref{FGcounituniversal} $\Leftrightarrow$ \ref{Funiversalarrow} is
dual to the proof of the equivalence \ref{FGunituniversal}
$\Leftrightarrow$ \ref{Guniversalarrow}. \\
[2mm] \noindent \ref{local} $\Leftrightarrow$ \ref{laxapproach}.
We first point out that the data of \ref{local} and \ref{laxapproach} are the same:
to obtain the outer maps of the span 2-cells for the horizontal natural isomorphism
in \ref{laxapproach}, we take $j$ and $k$ to be $1_X$ and $1_A$ and obtain bijections
$\bbA(FX,A)\cong\bbX(X,GA)$. To obtain the middle maps of the span 2-cells for
\ref{laxapproach}, we directly take the $\varphi_{j,k}$'s. Conversely, to obtain
the bijections $\varphi_{j,k}$ in \ref{local} from the horizontal natural isomorphism
in \ref{laxapproach}, we simply take the middle maps of the span 2-cells. So the data
of \ref{local} and \ref{laxapproach} are the same. As to the conditions: for the data
to form the horizontal natural transformation of \ref{laxapproach}, two compatibilities
are required: one horizontal compatibility equation for each square, which amounts
precisely to naturality of $\varphi_{j,k}$ in \ref{local}, and one compatibility
condition with respect to the coherence squares of the vertically lax double functors.
Since these coherence squares are given by vertical composition
(cf.~Example~\ref{exa:homsets_are_a_parametrized_presheaf}), this condition amounts
precisely to $\varphi$ being compatible with vertical composition.
\\

This completes the proof of the equivalence of
Definition~\ref{global} with each of \ref{FGunituniversal},
\ref{Guniversalarrow}, \ref{FGcounituniversal},
\ref{Funiversalarrow}, \ref{local}, and \ref{laxapproach}.
\end{proof}

We next prove a slightly weakened version of the equivalence Definition~\ref{global}$\Leftrightarrow$\ref{local}. The transpose of this slightly weakened version will be used in the proof of the {\it vertical} double adjunction between $\bEnd(\Sp)$ and $\bMnd(\Sp)$ in Proposition~\ref{freemonad_in_Span}.
\begin{thm}[Pseudo version of Theorem~\ref{thm:double_adjunction_descriptions}~\ref{local}] \label{thm:double_adjunction_def_iff_v_pseudo}
Let $\bbA$ and $\bbX$ be normal, vertically weak double categories. Let $F\colon \bbX \to \bbA$ and $G \colon \bbA \to \bbX$ be strict double functors, that is, $F$ and $G$ strictly preserve all compositions and identities of $\bbX$ respectively $\bbA$.  Then there exist strict horizontal natural transformations $\eta\co 1_{\bbX} \Rightarrow GF$ and $\varepsilon \co FG \Rightarrow 1_{\bbA}$ satisfying the two triangle identities if and only if statement \ref{local} of Theorem~\ref{thm:double_adjunction_descriptions} holds.
\end{thm}
\begin{proof}
The proof is the same as the proof of Definition~\ref{global} $\Leftrightarrow$ \ref{local} in
Theorem~\ref{thm:double_adjunction_descriptions}, only we must verify that the arguments there still make sense for the present hypotheses.

For the direction Definition~\ref{global} $\Rightarrow$ \ref{local}, we note i) the horizontal composition of squares is strictly associative (since the pseudo double categories are weak only vertically), ii) $G$ strictly preserves horizontal compositions, and iii) the interchange law holds in $\bbA$ and $\bbX$ as in any pseudo double category \cite[page 210]{Grandis-Pare:limits}.

For the direction \ref{local} $\Rightarrow$ Definition~\ref{global}, we note that $i^h_-$ is a horizontal identity square because $\bbA$ and $\bbX$ are normal (recall the discussion before Example~\ref{exa:span}).
\end{proof}

\medskip

In ordinary 1-category theory, a functor $F\co \bfA \rightarrow \bfX$
admits a right adjoint if and only if the presheaf $\bfA(F-,A)$ is
representable for each $A$. But for double categories and double
functors $F\co \bbA \rightarrow \bbX$, we must consider the
representability of the parameterized $\Sp^t$-valued presheaf $\bbA(F-,-)$.
We arrive at the following characterization of horizontal left double
adjoints in terms of parameterized representability.
\begin{thm}
  \label{thm:representability_characterization_of_dbladjs}
A double functor $F \co \bbX \rightarrow \bbA$ admits a horizontal right double adjoint if and only
if the parameterized presheaf on $\bbX$
$$\xymatrix{\bbA(F-,-)\co \bbX^{\text{\rm horop}} \times \bbV_1 \bbA \ar[r] & \Sp^t}$$
is represented by a double functor $G_0\co \bbV_1 \bbA \rightarrow \bbX$.
\end{thm}
\begin{rmk}
  Recalling the definition of $\bbV_1$ from Section~\ref{sec:Notation}, and the
  parametrized presheaves from~Definitions~\ref{def:parameterized_presheaf} and
  \ref{defn:parametrized_presheaf_representability},
  we see that
Theorem~\ref{thm:representability_characterization_of_dbladjs} essentially says
that a double functor $F$ admits a horizontal right double adjoint if and only
if for every vertical morphism $k$ in $\bbA$, the classical presheaf
$$\xymatrix@1{\bbA(F-,k)\co (\Ver \bbX, \Sq \bbX)^{\text{\rm op}} \ar[r] & \mathbf{Set}}$$
is representable in a way compatible with vertical composition.
\end{rmk}
\begin{proof}
Suppose that a horizontal right double adjoint $G$ exists. Then by
Theorem~\ref{thm:double_adjunction_descriptions} \ref{laxapproach}
the parameterized presheaves $\bbA(F-,-)$ and $\bbX(-,G-)$ are
horizontally naturally isomorphic as vertically lax functors on
$\bbX^{\text{\rm horop}} \times \bbA$, so their restrictions to
$\bbX^{\text{\rm horop}} \times \bbV_1 \bbA$ are also horizontally
naturally isomorphic. The double functor $G_0$ is simply the
restriction of $G$. We have represented $\bbA(F-,-)$ by $G_0$.

In the other direction, suppose that the parameterized presheaf on
$\bbX$
$$\xymatrix{\bbA(F-,-)\co \bbX^{\text{\rm horop}} \times \bbV_1 \bbA \ar[r] & \Sp^t}$$
is representable by a double functor $G_0\co
\bbV_1 \bbA \rightarrow \bbX$, and let $$\xymatrix{\varphi\co
\bbA(F-,-) \ar@{=>}[r] & \bbX(-,G_0-)}$$ be a horizontally natural
isomorphism between vertically lax functors. For vertical morphisms
$(j,k)$, we then have an isomorphism of spans in $\mathbf{Set}$.
$$\xymatrix@C=4pc{\bbA(Fs^vj,s^vj) \ar[r]^{\varphi(s^vj,s^vj)} & \bbX(s^vj,G_0s^vj) \\
\bbA(Fj,j) \ar[r]^{\varphi(j,k)} \ar[u]^{s^v} \ar[d]_{t^v} & \bbX(j,G_0j) \ar[u]_{s^v} \ar[d]^{t^v} \\
\bbA(Ft^vj,t^vj) \ar[r]_{\varphi(t^vj,t^vj)} & \bbX(t^vj,G_0t^vj)
}$$ Since $\bbV_1 \bbA$ has no nontrivial horizontal morphisms
or squares, the condition of horizontal naturality in $k$ is
satisfied vacuously. So, essentially we have horizontally natural
bijections $\varphi(-,k)\co \bbA(F-,k) \Rightarrow \bbX(-,G_0k)$,
and these correspond to universal squares from $F$ to $k$ of the form
$$\xymatrix{FG_0A \ar[d]_{FG_0 k} \ar[r]^-{\varepsilon(A)} \ar@{}[dr]|-{\varepsilon(k)} & A \ar[d]^k \\
FG_0 B \ar[r]_-{\varepsilon(B)} & B}$$ by Proposition~\ref{towardsrepresentability2}.
The assignments of $\varepsilon(A)$
and $\varepsilon(k)$ to $A$ and $k$ form a functor
$$\xymatrix{\varepsilon \co (\Obj \bbA, \Ver \bbA) \ar[r] & (\Hor \bbX, \Sq \bbX)}$$
because of the compatibility of $\varphi$ with the vertical laxness
of the parameterized presheaves. Finally, the characterization in
Theorem~\ref{thm:double_adjunction_descriptions}
\ref{Funiversalarrow} tells us that $G_0$ extends to a horizontal right adjoint
$G$, defined on squares $\alpha$ using universality and the equation
$\left[FG \alpha \;\; \varepsilon(t^h\alpha) \right] = \left[
\varepsilon(s^h\alpha) \;\; \alpha \right]$.
\end{proof}

\begin{rmk}\label{monoidal}
  In this section we have treated horizontal double adjunctions.  By
  transposition, all the results are equally valid for vertical double
  adjunctions.  In practice, however, the two notions are very different, as
  further properties or structure of the double categories in question may break
  the symmetry.  An instructive example is given by
  one-object/one-vertical-arrow double categories: these are monoids internal to
  $\mathbf{Cat}$, i.e.~monoidal categories (with strictness according to the
  strictness of the double categories).  Double functors between such are
  precisely monoidal functors (again with according strictness).
  Vertical natural transformations are precisely monoidal natural
  transformations.  Horizontal natural transformations are
  something quite different, some sort of intertwiners: for two double functors
  $F,G \co \bbD \to \bbC$ between one-object/one-vertical-arrow double
  categories, a horizontal natural transformation gives to a horizontal arrow
  $S$ of $\bbC$ (i.e.~an object of the corresponding monoidal category $\bfC$)
  and an equation (or $2$-cell) $S\otimes F = G \otimes S$ (where $\otimes$
  denotes horizontal composition, i.e.~the tensor product in $\bfC$).
\end{rmk}

\section{Compatibility with Foldings or Cofoldings}
\label{sec:Double_Adjunctions_with_Foldings_Cofoldings}

Many double categories of interest have additional structure that allows one to
reduce certain questions about the double category to questions about the horizontal
2-category.  There are several different, but closely related,
formalisms for this sort of situation,
cf.~Brown--Mosa~\cite{BrownMosa}, Brown--Spencer~\cite{BrownSpencer},
Fiore~\cite{Fiore:PseudoAlgebrasPseudoDoubleCategories},
Grandis--Par\'e~\cite{Grandis-Pare:limits},
Shulman~\cite{Shulman:Framed}; comparisons between the different formalisms
can be found in \cite{Fiore:PseudoAlgebrasPseudoDoubleCategories} and
\cite{Shulman:Framed}.  In
this section we investigate how the additional structure of {\em folding} or
{\em cofolding}
on double categories allows us to reduce
questions concerning adjunctions
to their horizontal 2-categories.

The notion of folding was introduced in
\cite{Fiore:PseudoAlgebrasPseudoDoubleCategories}, extending notions from
\cite{BrownMosa}.  A {\em folding} associates to every vertical morphism a
horizontal morphism in a way that gives a bijection between certain squares in
the double category and certain $2$-cells in the horizontal $2$-category.  The
precise definition is given below.  In Example~\ref{exa:span_admits_a_folding}),
we illustrate the folding for the double category of spans, which to a set map
(vertical morphism) $j:A\to C$ associates the span (horizontal morphism) $A
\overset{1^h_A}{\leftarrow} A \overset{j}{\rightarrow}C$.
The double category of spans was discussed in Example~\ref{exa:span}.

A folding can be seen as a kind of covariant action of the vertical $1$-category on the
horizontal $2$-category, a sort of pushforward operation;
see \cite[Section 4]{Fiore:PseudoAlgebrasPseudoDoubleCategories}.
A {\em cofolding} is
similar to a folding but constitutes instead a contravariant kind of action of the
vertical $1$-category on the horizontal $2$-category, a sort of pullback
operation.  In Example~\ref{exa:span_admits_a_cofolding}, we illustrate the cofolding
for the double category of spans, which to a vertical map $j:A\to C$ associates
the horizontal morphism $C \overset{j}{\leftarrow} A
\overset{1^h_A}{\rightarrow}A$.

Folding together with cofolding is equivalent to having a framing in the sense
of Shulman~\cite{Shulman:Framed}, the category of spans being an archetypical
example.  However, some important double categories admit either a folding or a
cofolding but not both, and it is necessary to study the two notions separately.
This is the case for the double categories of endomorphisms and monads,
$\bEnd(\bbD)$ and $\bMnd(\bbD)$, in
Section~\ref{sec:EndosAndMonadsInADoubleCategory}: if $\bbD$ admits a cofolding,
 then so do $\bEnd(\bbD)$ and $\bMnd(\bbD)$
(cf.~Proposition~\ref{prop:cofolding_on_D_induces_cofolding_on_Mnd(D)_and_on_End(D)}),
but the analogous statement for foldings does not seem to be true.

The main result in this section,
Proposition~\ref{prop:folding_double_adjunction_iff_horizontal_double_adjunction},
states that if $F$ and $G$ are double functors between double categories with
foldings, and $F$ and $G$ preserve the foldings, then $F$ and $G$ are
horizontally double adjoint if and only if the horizontal 2-functors $\bfH F$
and $\bfH G$ are 2-adjoint.  For the special case of quintet double categories,
which we characterize in terms of folding with fully faithful holonomy in
Lemma~\ref{ffhol=>quintet} and Proposition~\ref{prop:H_and_V_fully_faithful},
we establish stronger characterizations
of double adjunctions: briefly, all notions of adjunction agree in this case,
see Corollary~\ref{cor:fully_faithful_holonomy_implies_all_adjunctions_equivalent}.

We begin the detailed discussion of foldings and cofolding with the notion of
quintets.

\begin{examp}[Direct quintets] \label{exa:quintets}
  With a $2$-category $\mathbf{K}$ is associated a double category
  $\mathbb{Q}\bfK$,  called the double category of {\em direct quintets}:
  its objects are the objects of $\bfK$, horizontal and vertical morphisms are
  the morphisms of $\bfK$, and the squares
\begin{equation} \label{equ:square_alpha}
\begin{array}{c}
\xymatrix{A \ar[r]^f \ar[d]_j \ar@{}[dr]|\alpha & B \ar[d]^k \\ C \ar[r]_g & D}
\end{array}
\end{equation}
are the 2-cells
$\alpha\co k \circ f \Rightarrow g \circ j$ in $\bfK$. The horizontal 2-category of
$\mathbb{Q}\bfK$ is $\bfK$. The vertical
2-category of $\mathbb{Q}\bfK$ is $\bfK$ with the
2-cells reversed.  The terminology ``quintet'' is due to Ehresmann~\cite{EhresmannQuintets}
for the case $\mathbf{K}=\mathbf{Cat}$.  We add the word ``direct'' to
distinguish from the ``inverse quintets'' introduced in
Example~\ref{exa:twistedquintets},
as we shall need both variants.
\end{examp}

The double category $\bbQ\bfK$ is entirely determined by its horizontal
2-category,
in fact, a quintet square $\alpha$ is by
definition a 2-cell in $\bfK$ between appropriate composites of
boundary components of $\alpha$. Similarly, any double category with
folding, as in the following definition,
is determined by its vertical 1-category and horizontal
2-category in the sense that squares with a given boundary are in
bijective correspondence with 2-cells in the horizontal 2-category
between appropriate ``boundary composites''.

\begin{defn}\label{def:folding}
  (Cf.~Brown--Mosa~ \cite{BrownMosa} for the edge-symmmetric case and
  Fiore~\cite{Fiore:PseudoAlgebrasPseudoDoubleCategories} for the general case.)
  A {\em folding} on a double category $\bbD$ is
a double functor
$\Lambda\co\mathbb{D} \rightarrow
\mathbb{Q}\mathbf{H}\mathbb{D}$ which is the identity on the
horizontal 2-category $\mathbf{H}\mathbb{D}$ of $\mathbb{D}$ and is
fully faithful on squares.
We proceed to spell out the details.

A {\it folding on a double category $\mathbb{D}$} consists of the following.
\begin{enumerate}
\item
A 2-functor $\overline{(-)}\co(\mathbf{V}\mathbb{D})_0 \rightarrow \mathbf{H}\mathbb{D}$
which is the identity on objects.  Here, the notation $(\mathbf{V}\mathbb{D})_0$
denotes the vertical 1-category of $\bbD$.  In other words, to each vertical
morphism $j\co A \rightarrow C$, there is associated a horizontal morphism
$\overline{j}\co A \rightarrow C$ with the same domain and codomain in a
functorial way.  We call this 2-functor $j \mapsto \overline{j}$ the {\it
holonomy}, following the terminology of Brown-Spencer in \cite{BrownSpencer},
who first distinguished the notion.
\item
Bijections $\Lambda^{f,k}_{j,g}$ from squares in $\mathbb{D}$ with
boundary
\begin{equation} \label{equ:boundary1:folding}
\begin{array}{c}
\xymatrix@R=3pc@C=3pc{A \ar[r]^f \ar[d]_j & B \ar[d]^k \\ C \ar[r]_g
& D}
\end{array}
\end{equation}
to squares in $\mathbb{D}$ with boundary
\begin{equation} \label{equ:boundary2:folding}
\begin{array}{c}
\xymatrix@R=3pc@C=3pc{A \ar[r]^{[f \ \overline{k}]} \ar@{=}[d] &
D \ar@{=}[d]
\\ A \ar[r]_{[\overline{j} \ g]} & D.}
\end{array}
\end{equation}
\end{enumerate}
These bijections are required to satisfy the following axioms.
\begin{enumerate}
\item
$\Lambda$ is the identity if $j$ and $k$ are vertical identity
morphisms.
\item \label{def:folding:Lambdahorizontal} \label{def:folding:Lambdahorizontalequ}
$\Lambda$ preserves horizontal composition of squares, that is,
$$\Lambda\left(
\begin{array}{c}\xymatrix@R=4pc@C=4pc{A \ar[r]^{f_1} \ar[d]_j
\ar@{}[dr]|{\alpha} & B \ar[r]^{f_2} \ar[d]|{\tb{k}}
\ar@{}[dr]|{\beta} & C \ar[d]^{\ell} \\ D \ar[r]_{g_1} & E
\ar[r]_{g_2} & F}
\end{array}
 \right) \quad
=
\begin{array}{c}
\xymatrix@R=4pc@C=4pc{ A \ar[r]^{[f_1 \ f_2 \ \overline{\ell}]}
\ar@{=}[d] \ar@{}[dr]|{[i_{f_1}^v \ \Lambda(\beta)]} & F \ar@{=}[d]  \\
A \ar[r]|{[f_1 \ \overline{k} \ g_2]} \ar@{=}[d]
\ar@{}[dr]|{[\Lambda(\alpha) \ i_{g_2}^v]} & F \ar@{=}[d] \\ A
\ar[r]_{[\overline{j} \ g_1 \ g_2]} & F.}
\end{array}$$
\item \label{def:folding:Lambdavertical} \label{def:folding:Lambdaverticalequ}
$\Lambda$ preserves vertical composition of squares, that is,
$$\Lambda \left( \begin{array}{c} \xymatrix@R=4pc@C=4pc{A \ar[d]_{j_1} \ar[r]^f
\ar@{}[dr]|\alpha & B \ar[d]^{k_1} \\ C \ar[r]|{\lr{g}}
\ar@{}[dr]|\beta \ar[d]_{j_2} & D \ar[d]^{k_2} \\ E \ar[r]_h &
F,}\end{array} \right) \quad
 =
\begin{array}{c} \xymatrix@R=4pc@C=4pc{A \ar[r]^{[ f \
\overline{k}_1 \ \overline{k}_2]} \ar@{=}[d]
\ar@{}[dr]|{[\Lambda(\alpha) \ i_{\overline{k}_2}^v]} & F
\ar@{=}[d] \\
A \ar[r]|{[\overline{j}_1 \ g \ \overline{k}_2]}
\ar@{}[dr]|{[i_{\overline{j}_1}^v \ \Lambda(\beta)]} \ar@{=}[d]
& F \ar@{=}[d] \\ A \ar[r]_{[\overline{j}_1 \ \overline{j}_2 \
h] } & F.}
\end{array}
$$
\item \label{def:folding:Lambdaidentity} \label{def:folding:Lambdaidentityequ}
$\Lambda$ preserves identity squares, that is,
$$\Lambda\left(
\begin{array}{c}
\xymatrix@R=4pc@C=4pc{A \ar@{=}[r] \ar[d]_{j} \ar@{}[dr]|{i_j^h} & A \ar[d]^j \\
B \ar@{=}[r] & B}
\end{array} \right) \quad
=
\begin{array}{c}
\xymatrix@R=4pc@C=4pc{A \ar[r]^{\overline{j}}
\ar@{=}[d] \ar@{}[dr]|{i^v_{\overline{j}}} & B \ar@{=}[d] \\
A \ar[r]_{\overline{j}}  & B.}
\end{array}
$$
\end{enumerate}
\end{defn}

\begin{examp} \label{exa:span_admits_a_folding}
The double category $\Sp$ admits a folding. The holonomy is
$$\left( \begin{array}{c} \xymatrix@R=1.5pc{A \ar[d]_j \\ C} \end{array} \right)
\xymatrix{\ar@{|->}[r] &}
\left( A \overset{1^h_A}{\leftarrow} A \overset{j}{\rightarrow}C
\right)$$ and the folding is
$$\left( \begin{array}{c}
\xymatrix{A \ar[d]_j & \ar[l]_{f_0} Y \ar[d]^\alpha \ar[r]^{f_1} & B
\ar[d]^k \\ C & \ar[l]^{g_0} Z \ar[r]_{g_1} & D}
\end{array} \right)
\xymatrix{\ar@{|->}[r] &} \left(
\begin{array}{c}\xymatrix@C=3pc{A \ar@{=}[d] & Y \ar[l]_{f_0} \ar[r]^{k\circ f_1}
  \ar[d]^{(f_0,\alpha)} & D \ar@{=}[d] \\
A & \ar[l]^-{\pr_1} A \times_C Z \ar[r]_-{g_1 \circ \pr_2} &
D}\end{array} \right).$$
\end{examp}

\begin{rmk} \label{rmk:folding_implies_2cell_composition_compatibility}
If a double category $\bbD$ is equipped with a folding, then 2-cell
composition in the vertical 2-category $\bfV \bbD$ corresponds to
2-cell composition in the horizontal 2-category $\bfH \bbD$. More
precisely, if $f_1,f_2,g_1,g_2$ are identities in
Definition~\ref{def:folding}~\ref{def:folding:Lambdahorizontal},
then $\left[\; \alpha \;\; \beta \; \right]$ is the vertical
composition $\beta \odot \alpha$ in the 2-category $\bfV \bbD$, and
compatibility with horizontal composition says $\Lambda(\beta \odot
\alpha)=\Lambda(\alpha) \odot \Lambda(\beta)$. Concerning vertical
composition in the 2-category $\bfV \bbD$, if $f,g,h$ in
Definition~\ref{def:folding}~\ref{def:folding:Lambdavertical}, then
$\vcomp{\alpha}{\beta}$ is the horizontal composition $\beta *
\alpha$ in the 2-category $\bfV \bbD$, and $\Lambda (\beta * \alpha)
= \Lambda (\beta) * \Lambda(\alpha)$.
\end{rmk}

\begin{defn}[Compatibility with folding] \label{def:morphismwithfolding}
Let $\bbC$ and $\bbD$ be double categories with folding.
\begin{enumerate}
\item
A double functor $F\co \bbC \rightarrow \bbD$ is
{\it compatible with the foldings} if
$$F(\overline{j})=\overline{F(j)} \qquad \text{ and }
\qquad F(\Lambda^{\bbC}(\alpha))=\Lambda^{\bbD}(F(\alpha))$$
for all vertical morphisms $j$ and squares $\alpha$ in $\bbC$.
\item
Let $F,G\co \bbC \rightarrow \bbD$ be double functors compatible with the
foldings.  A horizontal natural transformation $\theta\co F \Rightarrow G$ is
{\it compatible with the foldings} if for all vertical morphisms $j$ in $\bbC$
the following equation holds.
\begin{equation}\label{equ:2cellwithfolding:horizontal_compatibility}
\Lambda\left(
\begin{array}{c}
\xymatrix@R=4pc@C=4pc{FA \ar[r]^{\theta A} \ar[d]_{Fj} \ar@{}[dr]|{\theta j}
& GA \ar[d]^{Gj} \\
FC \ar[r]_{\theta C} & GC}
\end{array} \right) \quad
=
\begin{array}{c}
\xymatrix@R=4pc@C=4pc{FA \ar[r]^{[\theta A \ G\overline{j}]}
\ar@{=}[d] \ar@{}[dr]|{i^v_{[\theta A \
G\overline{j}]}} & GC \ar@{=}[d] \\
FA \ar[r]_{[ F\overline{j} \ \theta C]} \ar[r] & GC}
\end{array}
\end{equation}
\item
Let $F,G\co \bbC \rightarrow \bbD$ be double functors
compatible with the foldings.
A vertical natural transformation $\sigma\co F \Rightarrow G$ is
{\it compatible with the foldings} if for all vertical morphisms
$j$ the following equation holds.
\begin{equation}\label{equ:2cellwithfolding:vertical_compatibility}
\Lambda\left(
\begin{array}{c}
\xymatrix@R=4pc@C=4pc{FA \ar[r]^{F \overline{j}} \ar[d]_{\sigma A}
\ar@{}[dr]|{\sigma \overline{j}} & FC \ar[d]^{\sigma C} \\
GA \ar[r]_{G \overline{j}} & GC}
\end{array} \right) \quad
=
\begin{array}{c}
\xymatrix@R=4pc@C=4pc{FA \ar[r]^{[F \overline{j} \ \overline{\sigma
C}]} \ar@{=}[d] \ar@{}[dr]|{i^v_{[F \overline{j} \ \overline{\sigma
C}]}} & GC \ar@{=}[d] \\
FA \ar[r]_{[\overline{\sigma A}  \ G\overline{j}]} \ar[r] & GC}
\end{array}
\end{equation}
\end{enumerate}
\end{defn}

\bigskip

Some double categories admit a cofolding rather than a folding, as the following
variant of the quintets of Example~\ref{exa:quintets} illustrates.  For double
categories of monads and endomorphisms (in the sense of
\cite{FioreGambinoKock:DoubleMonadsI} and
Section~\ref{sec:EndosAndMonadsInADoubleCategory} below),
cofoldings are more relevant than
foldings, since cofoldings are inherited from the underlying double category
(cf.~Proposition~\ref{prop:cofolding_on_D_induces_cofolding_on_Mnd(D)_and_on_End(D)})
whereas foldings are not.

\begin{examp}[Inverse quintets] \label{exa:twistedquintets}
For $\bfK$ a 2-category, the double category of {\em inverse quintets}
$\overline{\mathbb{Q}}\bfK$ is
the double category in which the objects are the objects of
$\bfK$, the horizontal 1-category is the underlying 1-category
of $\bfK$, the vertical 1-category is the {\it opposite} of
the underlying 1-category of $\bfK$, and the squares
$$\begin{array}{c} \xymatrix{A \ar[r]^f \ar[d]_{j^\text{op}}
\ar@{}[dr]|\alpha & B \ar[d]^{k^\text{op}} \\ C \ar[r]_g & D}
\end{array}$$ are 2-cells of the form $$\begin{array}{c} \xymatrix{A
\ar[r]^f \ar@{=>}[dr]^\alpha & B  \\ C \ar[r]_g \ar[u]^j & D
\ar[u]_k}\end{array}$$ in $\bfK$. The double category
$\overline{\mathbb{Q}}\bfK$ admits a {\it cofolding} in the
following sense.
\end{examp}

\begin{defn} \label{def:cofolding}
A {\em cofolding} is a double functor
$\Lambda\co\mathbb{D} \rightarrow
\overline\bbQ\mathbf{H}\mathbb{D}$ which is the identity on the
horizontal 2-category $\mathbf{H}\mathbb{D}$ of $\mathbb{D}$ and is
fully faithful on squares.  We proceed to spell out the details.

A {\it cofolding on a double category $\mathbb{D}$} consists of the following.
\begin{enumerate}
\item
A 2-functor $(-)^*\co(\mathbf{V}\mathbb{D})_0^{\text{op}} \rightarrow
\mathbf{H}\mathbb{D}$ which is the identity on objects.  Here, the notation
$(\mathbf{V}\mathbb{D})_0^{\text{op}}$ denotes the opposite of the vertical
1-category of $\bbD$.  In other words, to each vertical morphism $j\co A
\rightarrow C$, there is associated a horizontal morphism $j^*\co C \rightarrow
A $ in a functorial way.  We call the 2-functor $j \mapsto j^*$ the {\it
coholonomy}.
\item
Bijections $\Lambda^{f,k}_{j,g}$ from squares in $\mathbb{D}$ with
boundary
\begin{equation} \label{equ:boundary1:cofolding}
\begin{array}{c}
\xymatrix@R=3pc@C=3pc{A \ar[r]^f \ar[d]_j & B \ar[d]^k \\ C \ar[r]_g
& D}
\end{array}
\end{equation}
to squares in $\mathbb{D}$ with boundary
\begin{equation} \label{equ:boundary2:cofolding}
\begin{array}{c}
\xymatrix@R=3pc@C=3pc{C \ar[r]^{[j^* \ f ]} \ar@{=}[d] &
B \ar@{=}[d]
\\ C \ar[r]_{[g \ k^*]} & B.}
\end{array}
\end{equation}
\end{enumerate}
These bijections are required to satisfy the following axioms.
\begin{enumerate}
\item
$\Lambda$ is the identity if $j$ and $k$ are vertical identity
morphisms.
\item \label{def:cofolding:Lambdahorizontal}
\label{def:cofolding:Lambdahorizontalequ}
$\Lambda$ preserves horizontal composition of squares, that is,
$$\Lambda\left(
\begin{array}{c}\xymatrix@R=4pc@C=4pc{A \ar[r]^{f_1} \ar[d]_j
\ar@{}[dr]|{\alpha} & B \ar[r]^{f_2} \ar[d]|{\tb{k}}
\ar@{}[dr]|{\beta} & C \ar[d]^{\ell} \\ D \ar[r]_{g_1} & E
\ar[r]_{g_2} & F}
\end{array}
 \right) \quad
=
\begin{array}{c}
\xymatrix@R=4pc@C=4pc{ D \ar[r]^{[j^* \ f_1 \ f_2]}
\ar@{=}[d] \ar@{}[dr]|{[ \Lambda(\alpha) \ i_{f_2}^v ]} & C \ar@{=}[d]  \\
D \ar[r]|{[g_1 \ k^* \ f_2]} \ar@{=}[d]
\ar@{}[dr]|{[i_{g_1}^v \ \Lambda(\beta) ]} & C \ar@{=}[d] \\ D
\ar[r]_{[g_1 \ g_2 \ \ell^*]} & C.}
\end{array}$$
\item \label{def:cofolding:Lambdavertical} \label{def:cofolding:Lambdaverticalequ}
$\Lambda$ preserves vertical composition of squares, that is,
$$\Lambda \left( \begin{array}{c} \xymatrix@R=4pc@C=4pc{A \ar[d]_{j_1} \ar[r]^f
\ar@{}[dr]|\alpha & B \ar[d]^{k_1} \\ C \ar[r]|{\lr{g}}
\ar@{}[dr]|\beta \ar[d]_{j_2} & D \ar[d]^{k_2} \\ E \ar[r]_h &
F,}\end{array} \right) \quad
 =
\begin{array}{c} \xymatrix@R=4pc@C=4pc{E \ar[r]^{[j_2^* \ j_1^* f]} \ar@{=}[d]
\ar@{}[dr]|{[i_{j_2^*}^v \ \Lambda(\alpha)]} & B
\ar@{=}[d] \\
E \ar[r]|{[j_2^* \ g \ k_1^*]}
\ar@{}[dr]|{[\Lambda(\beta) \ i_{k_1^*}^v ]} \ar@{=}[d]
& B \ar@{=}[d] \\ E \ar[r]_{[h \ k_2^* \ k_1^*]} & B.}
\end{array}
$$
\item \label{def:cofolding:Lambdaidentity} \label{def:cofolding:Lambdaidentityequ}
$\Lambda$ preserves identity squares, that is,
$$\Lambda\left(
\begin{array}{c}
\xymatrix@R=4pc@C=4pc{A \ar@{=}[r] \ar[d]_{j} \ar@{}[dr]|{i_j^h} & A \ar[d]^j \\
B \ar@{=}[r] & B}
\end{array} \right) \quad
=
\begin{array}{c}
\xymatrix@R=4pc@C=4pc{B \ar[r]^{j^*}
\ar@{=}[d] \ar@{}[dr]|{i^v_{j^*}} & A \ar@{=}[d] \\
B \ar[r]_{j^*}  & A.}
\end{array}
$$
\end{enumerate}
\end{defn}

\begin{examp} \label{exa:span_admits_a_cofolding}
The double category $\Sp$ admits a cofolding. The coholonomy is
$$\left( \begin{array}{c} \xymatrix@R=1.5pc{A \ar[d]_j \\ C} \end{array} \right)
\xymatrix{\ar@{|->}[r] &}
\left( C \overset{j}{\leftarrow}  A \overset{1^h_A}{\rightarrow} A
\right)$$ and the cofolding is
$$\left( \begin{array}{c}
\xymatrix{A \ar[d]_j & \ar[l]_{f_0} Y \ar[d]^\alpha \ar[r]^{f_1} & B
\ar[d]^k \\ C & \ar[l]^{g_0} Z \ar[r]_{g_1} & D}
\end{array} \right)
\xymatrix{\ar@{|->}[r] &} \left(
\begin{array}{c}\xymatrix@C=3pc{C \ar@{=}[d] & Y \ar[l]_{j \circ f_0} \ar[r]^{f_1}
  \ar[d]^{(\alpha,f_1)} & B \ar@{=}[d] \\
C & \ar[l]^-{g_0 \circ \pr_1} Z \times_D B \ar[r]_-{\pr_2} &
B}\end{array} \right).$$
\end{examp}

\begin{defn}[Compatibility with cofolding] \label{def:morphismwithcofolding}
Let $\bbC$ and $\bbD$ be double categories with cofolding.
\begin{enumerate}
\item
A double functor $F\co \bbC \rightarrow \bbD$ is
{\it compatible with the cofoldings} if
$$F(j^*)=F(j)^* \qquad \text{ and } \qquad
F(\Lambda^{\bbC}(\alpha))=\Lambda^{\bbD}(F(\alpha))$$
for all vertical morphisms $j$ and squares $\alpha$ in $\bbC$.
\item
Let $F,G\co \bbC \rightarrow \bbD$ be double functors
compatible with the cofoldings.  A horizontal natural
transformation $\theta\co F \Rightarrow G$ is {\it
compatible with the cofoldings} if for all vertical morphisms $j$ in $\bbC$ the
following equation holds.
\begin{equation} \label{equ:2cellwithcofolding:horizontal_compatibility}
\Lambda\left(
\begin{array}{c}
\xymatrix@R=4pc@C=4pc{FA \ar[r]^{\theta A} \ar[d]_{Fj} \ar@{}[dr]|{\theta j} & GA \ar[d]^{Gj} \\
FC \ar[r]_{\theta C} & GC}
\end{array} \right) \quad
=
\begin{array}{c}
\xymatrix@R=4pc@C=4pc{FA \ar[r]^{[ Fj^* \ \theta A ]}
\ar@{=}[d] \ar@{}[dr]|{i^v_{[Fj^* \ \theta A ]}} & GC \ar@{=}[d] \\
FA \ar[r]_{[\theta C \ Gj^* ] } \ar[r] & GC}
\end{array}
\end{equation}
\item
Let $F,G\co \bbC \rightarrow \bbD$ be double functors
compatible with the cofoldings.
A vertical natural transformation $\sigma\co F \Rightarrow G$ is
{\it compatible with the cofoldings} if for all vertical morphisms
$j\colon A \to C$ the following equation holds.
\begin{equation}
\Lambda\left(
\begin{array}{c}
\xymatrix@R=4pc@C=4pc{FC \ar[r]^{F j^*} \ar[d]_{\sigma C}
\ar@{}[dr]|{\sigma j^*} & FA \ar[d]^{\sigma A} \\
GC \ar[r]_{G j^*} & GA}
\end{array} \right) \quad
=
\begin{array}{c}
\xymatrix@R=4pc@C=4pc{FC \ar[r]^{[(\sigma
C)^*  \ Fj^* ]} \ar@{=}[d] \ar@{}[dr]|{i^v_{[ (\sigma
C)^* \ F j^*   ]}} & GA \ar@{=}[d] \\
FC \ar[r]_{[ Gj^* \ (\sigma A)^* ]} \ar[r] & GA}
\end{array}
\end{equation}
\end{enumerate}
\end{defn}

We now come to the main result of this section.

\begin{prop}
\label{prop:folding_double_adjunction_iff_horizontal_double_adjunction}
Let $\bbA$ and $\bbX$ be double categories with folding (respectively cofolding)
and consider double functors $F$ and $G$ compatible with the foldings
(respectively cofoldings).
\begin{equation}
\xymatrix@C=4pc{\bbX \ar@/^1pc/[r]^{F} & \ar@/^1pc/[l]^{G} \bbA}
\end{equation}
Then $F$ and $G$ are horizontal double adjoints if and only if their horizontal
2-functors $\bfH F$ and $\bfH G$ are 2-adjoints.
 \end{prop}
\begin{proof}
If $F$ and $G$ are horizontal double adjoints, then $\bfH F$ and $\bfH G$ are 2-adjoints,
since the 2-functor $\bfH \co \mathbf{DblCat_h} \rightarrow \textbf{2-Cat}$
preserves adjoints, as does any 2-functor.

For the converse, suppose that $F$ and $G$ are compatible with the foldings and
$\varphi_{X,A}\co \bfH \bbA (FX,A) \rightarrow \bfH \bbX (X, GA)$ is a natural
isomorphism of categories.  We use the double adjunction characterization in
Theorem~\ref{thm:double_adjunction_descriptions}~\ref{local}.  For vertical
morphisms $j$ and $k$ in $\bbX$ and $\bbA$ respectively, we define a bijection
$$\xymatrix@1{\varphi_{j,k}\co \bbA(Fj,k) \ar[r] & \bbX(j,Gk)}$$
\begin{equation*}
\varphi_{j,k}(\alpha):=\left(\Lambda^{f^\dag,Gk}_{j,g^\dag} \right)^{-1}
\varphi_{sj, tk} \left( \Lambda^{f,k}_{Fj,g} (\alpha) \right).
\end{equation*}
Here $f^\dag$ and $g^\dag$ are the transposes of the horizontal morphisms $f$
and $g$ with respect to the underlying 1-adjunction.  The naturality of
$\varphi_{X,A}$ guarantees that the boundaries are correct.

The bijection $\varphi_{j,k}$ is compatible with vertical composition for the
following reasons:
\begin{enumerate}
\item \label{prop:folding_double_adjunction_iff_horizontal_double_adjunction:(i)}
$\varphi_{X,A}$ is compatible with the vertical composition of 2-cells in $\bfH \bbX$ and $\bfH\bbA$
\item \label{prop:folding_double_adjunction_iff_horizontal_double_adjunction:(ii)}
the isomorphism $\varphi_{X,A}$ is natural in $X$ and $A$, and
\item \label{prop:folding_double_adjunction_iff_horizontal_double_adjunction:(iii)}
the foldings are compatible with vertical composition as in
Definition~\ref{def:folding}~\ref{def:folding:Lambdavertical}.
\end{enumerate}

The naturality of $\varphi_{j,k}$ in $j$ and $k$ similarly follows
from
\ref{prop:folding_double_adjunction_iff_horizontal_double_adjunction:(i)}
and
\ref{prop:folding_double_adjunction_iff_horizontal_double_adjunction:(ii)}
above, and the compatibility of the foldings with horizontal
composition in
Definition~\ref{def:folding}~\ref{def:folding:Lambdahorizontal}.

These natural bijections $\varphi_{j,k}$ compatible with vertical
composition are equivalent to a unit $\eta$ and counit $\varepsilon$
in a horizontal double adjunction by
Theorem~\ref{thm:double_adjunction_descriptions}~\ref{local}, so we
are finished.

The analogous proof works for the cofolding claim.
\end{proof}

\begin{rmk}
In
Proposition~\ref{prop:folding_double_adjunction_iff_horizontal_double_adjunction},
note that the horizontal natural transformations $\eta$ and $\varepsilon$ which
make $F$ and $G$ into horizontal double adjoints are not required to be
compatible with the foldings, though if $\eta$ and $\varepsilon$ exist, they can
be replaced by horizontal natural transformations compatible with the foldings.
Note also that the holonomy (respectively coholonomy) is not required to be
fully faithful.
\end{rmk}

Proposition~\ref{prop:folding_double_adjunction_iff_horizontal_double_adjunction}
allows us to draw conclusions about horizontal double adjointness when both
double functors $F$ and $G$ are already given, and are compatible
  with the foldings.  It would be useful to have criteria for
  concluding the existence of a horizontal right double adjoint for a
  given double functor $F$ (compatible with foldings) given the
  existence of a right 2-adjoint for $\bfH F$, without referencing $G$
  at the outset.  One criterion that comes to mind is to require the
  holonomy to be fully faithful, but this happens only for double
  categories of direct quintets, as we now proceed to explain.  A subtler
  criterion for a special case of interest will be derived in
  Proposition~\ref{prop:bijection_for_fixed_underlying_morphism}.

\begin{examp}\label{quintets:ff}
  If $\bfK$ is a $2$-category,
  the canonical folding of the double category of direct quintets $\bbQ\bfK$
  of Example~\ref{exa:quintets} has fully faithful holonomy.  Similarly,
  the canonical cofolding on the double category of inverse quintets
  $\Qbar\bfK$ of Example~\ref{exa:twistedquintets} has fully faithful coholonomy.
\end{examp}

\begin{lem}\label{ffhol=>quintet}
  If $\bbD$ is a double category with folding and fully faithful holonomy,
  then the folding $\Lambda\co \bbD \to \bbQ \bfH \bbD$ is an isomorphism
  of double categories.
\end{lem}

\begin{proof}
  Indeed, $\Lambda$ is the identity on the horizontal 2-category,
  fully faithful on the vertical 1-category, and fully faithful on squares.
\end{proof}

\begin{lem}\label{VtoH}
  If $\bbD$ and $\bbC$ are double categories with fully faithful holonomy, and
  $F$ and $G$ are double functors $\bbD\to\bbC$ compatible with the holonomies,
  then
  the holonomy and
  folding provide a 1-1 correspondence between
  $2$-natural transformations $\bfV F \Rightarrow \bfV G$ and
  2-natural transformations $\bfH F \Rightarrow \bfH G$.
\end{lem}
\begin{proof}
  This is a consequence of the compatibility with horizontal
  composition of 2-cells in the vertical 2-category,
  cf.~Remark~\ref{rmk:folding_implies_2cell_composition_compatibility}.
\end{proof}

  In fact, we can refine Lemma~\ref{ffhol=>quintet} to an equivalence of
  $2$-categories.  Let $\mathbf{DblCatFoldHol_h}$ denote the 2-category of small double
  categories with folding and fully faithful holonomy, double functors
  compatible with foldings, and horizontal natural transformations
  compatible with folding (see Definitions~\ref{def:folding} and
  \ref{def:morphismwithfolding}).  Let $\mathbf{DblCatFoldHol_v}$ denote
  the 2-category of small double categories with folding and fully
  faithful holonomy, double functors compatible with foldings, and
  vertical natural transformations compatible with folding.

\begin{prop} \label{prop:H_and_V_fully_faithful}
The forgetful 2-functors
  \begin{align*} 
  &\xymatrix{\bfH \co \mathbf{DblCatFoldHol_h} \ar[r] & \mathbf{2Cat}}
\\  
&\xymatrix{\bfV \co \mathbf{DblCatFoldHol_v} \ar[r] & \mathbf{2Cat}}
  \end{align*}
  are equivalences of $2$-categories.
\end{prop}
\begin{proof}
    Note first that $\bfH$ and $\bfV$ are essentially surjective by
    Examples~\ref{exa:quintets} and \ref{quintets:ff}.
Suppose $F,G \co \bbC \rightarrow \bbD$ are double functors compatible with
foldings, and in particular compatible with the fully faithful holonomy, and
suppose $\bfH F =\bfH G$.  Then the double functors $F$ and $G$ agree on the
horizontal 2-categories.  If $j$ is a vertical morphism in $\bbC$, then
$\overline{F(j)}=F(\overline{j})=G(\overline{j})=\overline{G(j)}$, and
$F(j)=G(j)$ by the faithfulness of the holonomy.  The double functors $F$ and
$G$ similarly agree on squares because of the folding bijections.  Conversely,
if a 2-functor is defined on horizontal 2-categories, then it can be extended to
the double categories using the bijective holonomy and then the foldings.  Thus
$\bfH : \mathbf{DblCatFoldHol_h} \to \mathbf{2Cat}$ is bijective on the objects
of hom-categories.  Similarly, $\bfV$ is bijective on the objects of
hom-categories (here the fullness of the holonomy plays a role).

Similar arguments hold for injectivity on horizontal respectively vertical
natural transformations.

For fullness of $\bfH$ for 2-natural transformations, suppose $\theta \co \bfH F
\Rightarrow \bfH G$ is a 2-natural transformation.  We extend $\theta$ to a
horizontal natural transformation: for a vertical morphism $j$ in $\bbC$, define
$\theta j$ by equation \eqref{equ:2cellwithfolding:horizontal_compatibility}.
We verify double naturality for $\theta$, namely the equation $\left[\; F\alpha
\;\; \theta k \; \right] = \left[\; \theta j \;\; G \alpha \; \right]$ for any
square $\alpha$ in $\bbC$ with boundary as in equation \eqref{equ:square_alpha}.
By the definition of $\theta j$ and $\theta k$ via equation
\eqref{equ:2cellwithfolding:horizontal_compatibility}, we have $\Lambda(\theta
j)=i^v_{\left[\; \theta A \;\; G \overline{j} \; \right]}$ and $\Lambda(\theta
k)=i^v_{\left[\; \theta B \;\; G \overline{k} \; \right]}$, so that the equation
\begin{equation} \label{equ:prop:H_and_V_fully_faithful:H_surjective_on_2cells}
\begin{array}{c}
\begin{bmatrix}
i^v_{Ff} & \Lambda(\theta k) \\
\Lambda(F\alpha) & i^v_{\theta D}
\end{bmatrix}
\end{array}
=
\begin{array}{c}
\begin{bmatrix}
i^v_{\theta A} & \Lambda(G\alpha) \\
\Lambda(\theta j) & i^v_{G g}
\end{bmatrix}
\end{array}
\end{equation}
holds by 2-naturality of $\theta$. The double naturality then follows from
an application of $\Lambda^{-1}$ to
\eqref{equ:prop:H_and_V_fully_faithful:H_surjective_on_2cells} using
axiom~\ref{def:folding:Lambdahorizontal} of Definition~\ref{def:folding}.

For fullness of $\bfV$ on 2-natural transformations, suppose $\sigma \co \bfV F
\Rightarrow \bfV G$ is a 2-natural transformation.  We extend $\sigma$ to a
vertical natural transformation: for any horizontal morphism $\overline{j}$ in
$\bbC$, define $\sigma \overline{j}$ by equation
\eqref{equ:2cellwithfolding:vertical_compatibility}.  Recall that the holonomy
is fully faithful, so any horizontal morphism is of the form $\overline{j}$ for
a unique vertical morphism $j$.  The proof for surjectivity of $\bfV$ on
2-natural transformations proceeds like that of $\bfH$, using Lemma~\ref{VtoH}.
\end{proof}

\begin{cor} \label{cor:fully_faithful_holonomy_implies_all_adjunctions_equivalent}
Let $\bbA$ and $\bbX$ be double categories with folding and fully
faithful holonomies. Let $F \co \bbX \rightarrow \bbA$
be a double functor compatible with the foldings. Then the following
are equivalent.
\begin{enumerate}
\item \label{cor:fully_faithful_holonomy_implies_all_adjunctions_equivalent:horizontal_double}
The double functor $F$ admits a horizontal right double adjoint (not necessarily
compatible with the foldings).
\item \label{cor:fully_faithful_holonomy_implies_all_adjunctions_equivalent:horizontal_2}
The 2-functor $\bfH F \co \bfH \bbX \rightarrow \bfH \bbA$ admits a right 2-adjoint.
\item \label{cor:fully_faithful_holonomy_implies_all_adjunctions_equivalent:vertical_double}
The double functor $F$ admits a vertical right double adjoint
(not necessarily compatible with the foldings).
\item \label{cor:fully_faithful_holonomy_implies_all_adjunctions_equivalent:vertical_2}
The 2-functor $\bfV F \co \bfV \bbX \rightarrow \bfV \bbA$ admits a right 2-adjoint.
\end{enumerate}
\end{cor}
\begin{proof}
By Proposition~\ref{prop:H_and_V_fully_faithful}, the 2-functor
$\bfH : \mathbf{DblCatFoldHol_h} \to \mathbf{2Cat}$ is 2-fully
faithful, so $F$ admits a horizontal right double adjoint compatible
with the foldings if and only if $\bfH F$ admits a right 2-adjoint.
But if $F$ admits a horizontal right double adjoint $G$ not
necessarily compatible with the foldings, then $\bfH G$ is still a
right 2-adjoint to $\bfH F$, and
Proposition~\ref{prop:H_and_V_fully_faithful} applies to extend the
2-adjunction $\bfH F \dashv \bfH G$ to a horizontal double
adjunction with horizontal left double adjoint $F$.  Thus
\ref{cor:fully_faithful_holonomy_implies_all_adjunctions_equivalent:horizontal_double}$\Leftrightarrow$\ref{cor:fully_faithful_holonomy_implies_all_adjunctions_equivalent:horizontal_2}
and similarly
\ref{cor:fully_faithful_holonomy_implies_all_adjunctions_equivalent:vertical_double}$\Leftrightarrow$\ref{cor:fully_faithful_holonomy_implies_all_adjunctions_equivalent:vertical_2}.

To complete the proof, we observe
\ref{cor:fully_faithful_holonomy_implies_all_adjunctions_equivalent:horizontal_2}$\Leftrightarrow$\ref{cor:fully_faithful_holonomy_implies_all_adjunctions_equivalent:vertical_2},
because the fully faithful holonomy and folding provide a 1-1 correspondence
between 2-natural transformations $\bfV F_1 \Rightarrow \bfV F_2$ and 2-natural
transformations $\bfH F_1 \Rightarrow \bfH F_2$, by Lemma~\ref{VtoH}.
\end{proof}

For completeness, we also state the analogues of Lemma~\ref{ffhol=>quintet},
Proposition~\ref{prop:H_and_V_fully_faithful} and
Corollary~\ref{cor:fully_faithful_holonomy_implies_all_adjunctions_equivalent}
for double categories with cofoldings and fully faithful
coholonomies.

\begin{lem}\label{ffcoh=>quintet}
  If $\bbD$ is a double category with cofolding and fully faithful coholonomy,
  then the cofolding $\Lambda\co \bbD \to \overline{\bbQ} \bfH \bbD$ is an isomorphism
  of double categories.
\end{lem}
With self-explanatory notation as in
Proposition~\ref{prop:H_and_V_fully_faithful}, we have:
\begin{prop} \label{prop:H_and_V_fully_faithful_cofolding_case}
The forgetful 2-functors
  \begin{align*}
     & \xymatrix{\bfH \co \mathbf{DblCatCofoldCohol_h} \ar[r] &
  \mathbf{2Cat}}\\
     & \xymatrix{\bfV \co \mathbf{DblCatCofoldCohol_v}^{\rm co} \ar[r] &
\mathbf{2Cat}}
  \end{align*}
  are equivalences of $2$-categories.
\end{prop}
\noindent
The reversal of $2$-cells by $\bfV$ (indicated with the superscript ${}^{\rm
co}$) stems from the contravariant nature of the cofolding.

\begin{proof}
The entire proof is very similar to that of
Proposition~\ref{prop:H_and_V_fully_faithful}.  The only small difference is in
the fullness of $\bfH$ and $\bfV$ for 2-natural
transformations.  Suppose $\theta \co \bfH F \Rightarrow \bfH G$ is a 2-natural
transformation.  We extend $\theta$ to a horizontal natural transformation: for
a vertical morphism $j$ in $\bbC$, define $\theta j$ by equation
\eqref{equ:2cellwithcofolding:horizontal_compatibility}.  By the definition of
$\theta j$ and $\theta k$ via equation
\eqref{equ:2cellwithcofolding:horizontal_compatibility}, we have $\Lambda(\theta
j)=i^v_{\left[\; Fj^* \;\; \theta A \; \right]}$ and $\Lambda(\theta
k)=i^v_{\left[\; Fk^* \;\; \theta B \; \right]}$, so that the equation
\begin{equation} \label{equ:prop:H_and_V_fully_faithful:H_surjective_on_2cells:cofolding_case}
\begin{array}{c}
\begin{bmatrix}
\Lambda(F\alpha) & i^v_{\theta B} \\ i^v_{Fg} & \Lambda(\theta k)
\end{bmatrix}
\end{array}
=
\begin{array}{c}
\begin{bmatrix}
\Lambda(\theta j) & i^v_{Gf} \\ i^v_{\theta C} & \Lambda(G\alpha)
\end{bmatrix}
\end{array}
\end{equation}
holds by 2-naturality of $\theta$.  The double naturality equation $\left[\;
F\alpha \;\; \theta k \; \right] = \left[\; \theta j \;\; G \alpha \; \right]$
for $\theta$ then follows from an application of $\Lambda^{-1}$ to
\eqref{equ:prop:H_and_V_fully_faithful:H_surjective_on_2cells:cofolding_case}
using axiom~\ref{def:cofolding:Lambdahorizontal} of
Definition~\ref{def:cofolding}.
\end{proof}

The contravariant nature of the cofolding also affects
the direction of the vertical adjunction in the following
cofolding analog of
Corollary~\ref{cor:fully_faithful_holonomy_implies_all_adjunctions_equivalent}:
%

\begin{cor}
\label{cor:fully_faithful_coholonomy_implies_all_adjunctions_equivalent}
  Let $\bbA$ and $\bbX$ be double categories with cofolding and fully
  faithful coholonomies.  Let $F \co \bbX \rightarrow \bbA$ be a
  double functor compatible with the cofoldings.  Then the following
  are equivalent.
\begin{enumerate}
  \item
  \label{cor:fully_faithful_coholonomy_implies_all_adjunctions_equivalent:horizontal_double}
The double functor $F$ admits a horizontal right double adjoint (not
necessarily compatible with the cofoldings).
  \item
  \label{cor:fully_faithful_coholonomy_implies_all_adjunctions_equivalent:horizontal_2}
  The 2-functor $\bfH F \co \bfH \bbX \rightarrow \bfH \bbA $ admits a
  right 2-adjoint.
  \item
  \label{cor:fully_faithful_coholonomy_implies_all_adjunctions_equivalent:vertical_double}
The double functor $F$ admits a vertical left double adjoint (not
necessarily compatible with the cofoldings).
  \item
  \label{cor:fully_faithful_coholonomy_implies_all_adjunctions_equivalent:vertical_2}
  The 2-functor $\bfV F \co \bfV \bbX \rightarrow \bfV \bbA$ admits a
  left 2-adjoint.
\end{enumerate}
\end{cor}

\bigskip

\section{Endomorphisms and Monads in a Double Category}
\label{sec:EndosAndMonadsInADoubleCategory}

The notions of endomorphism and monad in a double category were introduced in
\cite{FioreGambinoKock:DoubleMonadsI}, the main theorem of which gave
sufficient conditions for the existence of free monads in a double category.
One of the goals of this paper is to simultaneously remove several hypotheses
from our main theorem in~\cite{FioreGambinoKock:DoubleMonadsI} and strengthen
its conclusion to obtain
Theorem~\ref{thm:existence_of_free_monads} of this
paper, which says that if a double category $\bbD$ with cofolding admits the
construction of free monads in its horizontal 2-category, then $\bbD$ admits the
construction of free monads as a double category.  Towards that goal, we prove
in this section that a cofolding on $\bbD$ induces a cofolding on the double
categories $\bEnd(\bbD)$ and $\bMnd(\bbD)$ of endomorphisms and monads in
$\bbD$, see \cite[Definitions~2.3~and~2.4]{FioreGambinoKock:DoubleMonadsI}.
Another goal of this paper is
Theorem~\ref{thm:EilenbergMooreExistence}, the characterization of the existence
of Eilenberg--Moore objects in a double category in terms of representability of
certain parameterized presheaves.  For that we also need an understanding of the
double category $\bMnd(\bbD)$.

\bigskip

Following \cite{FioreGambinoKock:DoubleMonadsI}, by endomorphism and
monad in a double category we mean horizontal endomorphism and horizontal
monad.  Hence an {\em endomorphism} in a double category is a pair $(X,P)$ where $X$
is an object and $P:X\to X$ is a horizontal morphism.  A {\em monad structure} on
$(X,P)$ consists of squares
\[
\xymatrix{
X \ar[r]^{P} \ar@{=}[d]  \ar@{}[drr]|{\mu_P} & X \ar[r]^{P} & X \ar@{=}[d] \\
X \ar[rr]_{P} & & X } \qquad \xymatrix{
X \ar@{=}[r]   \ar@{=}[d]  \ar@{}[dr]|{\eta_P}  & X \ar@{=}[d] \\
X \ar[r]_{P} & X }
\]
satisfying obvious laws of associativity and unitality.  In other words,
endomorphisms and monads are the same as endomorphisms and monads in the
horizontal $2$-category.

A {\em horizontal map} between endomorphisms $(X,P)$ and $(Y,Q)$
is a horizontal morphism $F \co X \rightarrow Y$ together with a square
\begin{equation}\label{monadmap}
\begin{array}{c}
  \xymatrix{
X \ar[r]^{F} \ar@{=}[d]   \ar@{}[drr]|{\phi}   & Y \ar[r]^{Q} & Y \ar@{=}[d]  \\
X \ar[r]_{P} & X \ar[r]_{F} & Y .}
\end{array}
\end{equation}
A {\em vertical map} $(u, \bar{u}) : (X,P) \rightarrow (X',P')$ consists of a
vertical morphism $u \co X \rightarrow X'$ and a square
\[
\xymatrix{
X \ar[r]^{P} \ar[d]_{u} \ar@{}[dr]|{\bar{u}}   & X \ar[d]^{u} \\
X'  \ar[r]_{P'} & X' .}
\]
The definitions of horizontal and vertical maps between monads are similar,
but the squares $\phi$ and $\bar u$ are then subject to some evident compatibility
conditions with respect to the monad structures.
There are also notions of endomorphism square and monad square (which we shall
not recall here) making $\bEnd(\bbD)$ and $\bMnd(\bbD)$ into
double categories, cf.~\cite{FioreGambinoKock:DoubleMonadsI}.  See
Examples~\ref{exa:End(Span)} and \ref{exa:Mnd(Span)}.

The direction of the square $\phi$ in the definition of horizontal endomorphism
map and horizontal monad map is chosen
so as to agree with the convention of Street~\cite{Street:formal-monads} for
endomorphism maps and monad maps in the horizontal $2$-category, which in turn
is motivated among other things by the desire to pullback algebras
for monads.  This choice has some consequences for some other choices in this
paper, and we pause to explain this.  For brevity we talk only about monads, the
case of endomorphisms being analogous.

The other natural choice for horizontal monad maps
$(X,P) \to (Y,Q)$ is with squares of the form
\[
\xymatrix{
X \ar[r]^{P} \ar@{=}[d]   \ar@{}[drr]|{\phi}   & X \ar[r]^{F} & Y \ar@{=}[d]  \\
X \ar[r]_{F} & Y \ar[r]_{Q} & Y ,}
\]
which for fun we call Avenue monad maps in the following discussion.  We
temporarily denote
by $\bMnd^{\mathrm{st}}(\bbD)=\bMnd(\bbD)$
the double category whose horizontal morphisms are
Street monad maps (the convention used elsewhere in this paper), and by
$\bMnd^{\mathrm{av}}(\bbD)$ the double category with Avenue monad maps.
The two double categories have the same vertical morphisms.

Both notions of monad map refer only to the horizontal $2$-category and make
sense already for $2$-categories, so for a $2$-category $\bfK$ we have two
different $2$-categories of monads, $\Mnd^{\mathrm{st}}(\bfK)$ and
$\Mnd^{\mathrm{av}}(\bfK)$.  The two notions of monad maps for
$2$-categories can be combined into a
single double category that has Street monad maps as horizontal morphisms and
Avenue monad maps as vertical morphisms; there is a unique natural choice of
what square should be taken to be to make this into a double category.  This
double category is naturally isomorphic to $\bMnd^{\mathrm{st}}(\bbQ\bfK)$, which
is different from
$\bbQ(\Mnd^{\mathrm{st}}(\bfK))$: both double categories have
$\Mnd^\mathrm{st}(\bfK)$ as horizontal $2$-category, but while the vertical
$2$-category of $\bMnd^{\mathrm{st}}(\bbQ\bfK)$ is $\Mnd^{\mathrm{av}}(\bfK)$ with
$2$-cells reversed, the vertical $2$-category of $\bbQ(\Mnd^{\mathrm{st}}(\bfK))$
is $\Mnd^\mathrm{st}(\bfK)$ with the $2$-cells reversed.
In contrast we have the following result, whose
proof is a straightforward but tedious verification.

\begin{lem}\label{Qbar-Mnd}
  For any $2$-category $\bfK$, we have natural identifications
  \begin{xalignat*}{3}
    \bEnd^{\mathrm{st}}( \Qbar(\bfK) )  &= \Qbar( \End^{\mathrm{st}}(\bfK) )
   \  & \ \bMnd^{\mathrm{st}}( \Qbar(\bfK) )  &= \Qbar(
    \operatorname{Mnd}^{\mathrm{st}}(\bfK) ) \\
    \bEnd^{\mathrm{av}}( \bbQ(\bfK) ) &= \bbQ( \End^{\mathrm{av}}(\bfK) )
    \ & \ \bMnd^{\mathrm{av}}( \bbQ(\bfK) ) &= \bbQ(
    \Mnd^{\mathrm{av}}(\bfK) ) .
    \end{xalignat*}
\end{lem}

The fact that Street monad maps are more compatible with the inverse quintet
construction $\Qbar$ of Example~\ref{exa:twistedquintets}
than with the direct quintet construction $\bbQ$
(Example~\ref{exa:quintets})
explains to some extent
why in the following it is cofolding rather than folding that goes well with
monads.  With the Avenue convention on monad maps, the following results would
have concerned folding instead of cofolding.

\bigskip

The following is the main point of this section: a cofolding on a double
category $\bbD$ induces a
cofolding on $\bMnd(\bbD)$ and $\bEnd(\bbD)$.

\begin{prop} \label{prop:cofolding_on_D_induces_cofolding_on_Mnd(D)_and_on_End(D)}
If $(\bbD,\Lambda^\bbD)$ is a double category with cofolding, then the double
categories $\bMnd(\bbD)$ and $\bEnd(\bbD)$ inherit cofoldings from $\bbD$, and
the forgetful double functor $U\co \bMnd(\bbD) \rightarrow \bEnd(\bbD)$
preserves them.
\end{prop}
\begin{proof}
  We first construct the cofolding on $\bEnd(\bbM)$:
  if $(u, \bar{u})\co (X, P) \rightarrow (X',P')$ is a vertical
endomorphism map, then the corresponding horizontal endomorphism map
$(u, \bar{u})^*$ under the coholonomy is
\[
(u^*, \Lambda^\bbD(\bar{u}))\co (X', P') \rightarrow (X,P) ,
\]
if $\alpha$ is an endomorphism square, then the corresponding
endomorphism 2-cell is the $\bbD$-cofolding of $\alpha$, namely
$\Lambda^\bbD(\alpha)$.
It is straightforward to check, using  the functoriality of the coholonomy on
$\bbD$ and the compatibility of $\Lambda^\bbD$ with horizontal and vertical
composition of squares, that these assignments constitute a cofolding
on $\bEnd(\bbD)$.

Next we verify that the same construction of the cofolding works for monads:
if
$(X,P)$ and $(X',P')$ are monads, and $(u, \bar{u})$ is vertical
monad map, then $(u, \bar{u})^*=(u^*, \Lambda^\bbD(\bar{u}))$ is a
horizontal monad map, and if
$\alpha$ is a monad square, then $\Lambda^\bbD(\alpha)$ is a monad
$2$-cell.
This  follows readily from the compatibility of $\Lambda^\bbD$ with horizontal and
vertical composition of squares.  Since the two cofoldings are given by the
same construction, it is clear that the forgetful functor preserves them.
%
\end{proof}

In
Proposition~\ref{prop:cofolding_on_D_induces_cofolding_on_Mnd(D)_and_on_End(D)},
note that if $\bbD$ has fully faithful coholonomy, then the induced coholonomies
on $\bMnd(\bbD)$ and $\bEnd(\bbD)$ are again fully faithful.  This follows from
Lemma~\ref{ffcoh=>quintet} and Lemma~\ref{Qbar-Mnd}.  We have seen in
Corollary~\ref{cor:fully_faithful_coholonomy_implies_all_adjunctions_equivalent}
that when the coholonomy is fully faithful, all questions about adjunction can
be settled in the horizontal $2$-category, but we noted also that this
requirement is a very restrictive condition.  The following technical result can
be interpreted as saying that in the situation of the preceding proposition,
although $\bEnd(\bbD)$ and $\bMnd(\bbD)$ do not often have fully faithful
coholonomies, they do have some fully faithfulness relative to $\bbD$: for a
{\em fixed } vertical morphism $u$ in $\bbD$, we do get certain bijections.
This result, which generalizes \cite[Lemma~3.4]{FioreGambinoKock:DoubleMonadsI},
will play an important role in the proofs of
Proposition~\ref{prop:free_to_horizontal_adj} and
Theorem~\ref{thm:existence_of_free_monads}.

\begin{prop}
  \label{prop:bijection_for_fixed_underlying_morphism}
  In the situation of
  Proposition~\ref{prop:cofolding_on_D_induces_cofolding_on_Mnd(D)_and_on_End(D)},
  if $u\co X \rightarrow X'$ is a fixed vertical morphism
in $\bbD$, then
\[
(u, \bar{u}) \mapsto (u^*, \Lambda^{\bbD}(\bar{u}) )
\]
is a bijection between vertical endomorphism maps $(X,P) \rightarrow
(X',P')$ with underlying vertical morphism $u$ and horizontal
endomorphism maps $(X',P') \rightarrow (X,P)$ with underlying
horizontal morphism $u^*$. If $(X,P)$ and $(X',P')$ are monads, we
have a similar bijection between vertical monad maps
with underlying morphism $u$ and horizontal monad maps with
underlying morphism $u^*$.
\end{prop}
\begin{proof}
  Vertical endomorphism maps over $u$ from $(X,P)$ to $(X',P')$ are squares
  \[
\xymatrix{
X \ar[r]^{P} \ar[d]_{u} \ar@{}[dr]|{\bar{u}}   & X \ar[d]^{u} \\
X'  \ar[r]_{P'} & X' ,}
\]
  which under $\Lambda^{\bbD}$ correspond to squares
\[
\xymatrix{
X' \ar[r]^{u^*} \ar@{=}[d]   \ar@{}[drr]|{\Lambda^{\bbD}(\bar u)}   & X \ar[r]^{P} & X \ar@{=}[d]  \\
X' \ar[r]_{P'} & X' \ar[r]_{u^*} & X .}
\]
  which are precisely the horizontal endomorphism maps over $u^*$ from $(X',P')$ to
  $(X,P)$.  The assertion about monad maps is similar.
\end{proof}

\section{Example: Endomorphisms and Monads in $\Sp$}
\label{sec:DoubleAdjunctionBetweenEndosInSpanAndMonadsInSpan}

We consider the normal, horizontally weak double category $\Sp$ of spans in $\mathbf{Set}$ from
Example~\ref{exa:span} in order to exemplify the notions of endomorphism and monad in a double category,
to illustrate the local description of double adjunctions in Theorem~\ref{thm:double_adjunction_def_iff_v_pseudo} (a slightly weak version of Theorem~\ref{thm:double_adjunction_descriptions}~\ref{local}), and to motivate Theorem~\ref{thm:existence_of_free_monads} below. We establish by hand the following result, which is a special
case of \cite[Proposition 3.8]{FioreGambinoKock:DoubleMonadsI}.
\begin{prop}\label{freemonad_in_Span}
  The forgetful double functor $G:\bMnd(\Sp) \to \bEnd(\Sp)$ has a vertical
  double left adjoint $F$,
\begin{equation} \label{equ:DiGraph_Cat_Double_Adjunction}
\xymatrix@C=4pc{\bEnd(\Sp) \ar@/^1pc/[r]^{F} \ar@{}[r]|{\perp} &
\ar@/^1pc/[l]^{G} \bMnd(\Sp).}
\end{equation}
\end{prop}
Note that although
$\bEnd(\Sp)$ and $\bMnd(\Sp)$ are
horizontally weak double categories, the double
functors $F$ and $G$ strictly preserve all compositions and identities.
The 1-adjunction $$\xymatrix@C=4pc{\mathbf{DirGraph}
\ar@/^1pc/[r]^{Free} \ar@{}[r]|{\perp} & \ar@/^1pc/[l]^{Forget}
\mathbf{Cat}}$$ is the {\it vertical} 1-category part of
\eqref{equ:DiGraph_Cat_Double_Adjunction}.

\medskip

We next spell out the double categories $\bEnd(\Sp)$ and
$\bMnd(\Sp)$.

\begin{examp}[Endomorphisms in $\Sp$] \label{exa:End(Span)}
Objects and vertical morphisms of $\bEnd(\Sp)$ are directed graphs $G_0
\leftarrow G_1 \rightarrow G_0$ and morphisms of directed graphs. A
horizontal morphism $(U,\phi)\co G_* \rightarrow G_*'$  in
$\bEnd(\Sp)$ is a span $U \co G_0 \leftarrow U_1 \rightarrow G_0'$
equipped with a chosen (not necessarily vertically invertible)
square in $\Sp$ as below.\footnote{If $U \co G_0 \leftarrow U_1 \rightarrow G_0'$ is not an identity span, then
a square as in \eqref{equ:horizontal_morphism_in_End(Span)} is a (not necessarily bijective) function $\phi\co
U_1 \times_{G_0'} G_1' \rightarrow G_1 \times_{G_0} U_1$ making the relevant squares commute. If $U$ is an identity span, then a square as in \eqref{equ:horizontal_morphism_in_End(Span)} is a (not necessarily bijective) function $\phi\co G_1' \to G_1$. Recall the choice of pullback described in Example~\ref{exa:span}.}
\begin{equation} \label{equ:horizontal_morphism_in_End(Span)}
\begin{array}{c}
\xymatrix{G_0 \ar@{=}[d] & \ar[l] U_1 \ar[r] & G_0' \ar@{}[d]|\phi &
\ar[l] G_1' \ar[r] & G_0' \ar@{=}[d] \\
G_0 & \ar[l] G_1 \ar[r] & G_0 & \ar[l] U_1 \ar[r] & G_0' .}
\end{array}
\end{equation}
Horizontal composition of horizontal morphisms is by pullback, with the usual choice made for identities as in Example~\ref{exa:span} ($\phi$ is then the identity on $G_1$). The
associated $\phi$-part of the composite is the vertical composite
of the following squares.
\begin{equation} \label{equ:phi_composition}
  \begin{array}{c}
\xymatrix{G_0 \ar@{=}[d]
\ar@{}[drr]|{1_U} & \ar[l] U_1 \ar[r] & G_0' \ar@{=}[d] & \ar[l] V_1
\ar[r] & G_0'' \ar@{}[d]|\psi & \ar[l] G_1'' \ar[r] & G_0''
\ar@{=}[d]
\\ G_0 \ar@{=}[d] & \ar[l] U_1 \ar[r] & G_0' \ar@{}[d]|\phi &
\ar[l] G_1' \ar[r] & G_0'  \ar@{=}[d] \ar@{}[drr]|{1_V} &
\ar[l] V_1 \ar[r] & G_0'' \ar@{=}[d]
\\ G_0 & \ar[l] G_1 \ar[r] & G_0 & \ar[l] U_1 \ar[r] & G_0' &
\ar[l] V_1 \ar[r] & G_0'' .}
\end{array}
\end{equation}
A square in $\bEnd(\Sp)$
\begin{equation} \label{equ:square_in_End(Span)}
  \begin{array}{c}
\xymatrix{G_* \ar[r]^U \ar[d]_{J_*} \ar@{}[dr]|{\alpha} & G_*'
\ar[d]^{J_*'} \\ H_* \ar[r]_{V} & H_*'}
\end{array}
\end{equation}
is a square in $\Sp$
$$\xymatrix{G_0 \ar[d] & \ar[l] U_1 \ar[d]_\alpha \ar[r] & G_0' \ar[d] \\
H_0 & \ar[l] V_1 \ar[r] & H_0'}$$
such that the cube with $\phi$ on top and $\phi'$ on bottom
commutes. Horizontal and vertical composition of squares in
$\bEnd(\Sp)$ are the horizontal and vertical compositions of the
underlying squares in $\Sp$, for example, horizontal composition is
defined via pullback.
\end{examp}

\begin{examp}[Monads in $\Sp$] \label{exa:Mnd(Span)}
Objects and vertical
morphisms of $\bMnd(\Sp)$ are categories and functors.  The horizontal morphisms of $\bMnd(\Sp)$
are the same as Street's morphisms of monads in a 2-category
\cite{Street:formal-monads}.  Namely, a horizontal monad morphism $U\co C_*
\rightarrow D_*$ is a span $C_0 \leftarrow U_1 \rightarrow D_0$ and a square in
$\Sp$
$$\xymatrix{C_0 \ar@{=}[d] & \ar[l] U_1 \ar[r] & D_0 \ar@{}[d]|\phi & \ar[l] D_1 \ar[r] & D_0 \ar@{=}[d] \\
C_0 & \ar[l] C_1 \ar[r] & C_0 & \ar[l] U_1 \ar[r] & D_0}$$ such that
$$\vcomp{\left[\; 1^v_U \;\; \eta^D \; \right]}{\phi}=\left[\; \eta^C \;\; 1^v_U \;\right]$$
and
$$\begin{bmatrix}
& \phi & & 1^v_D \\
1^v_C & & \phi &\\
& \mu^C & & 1^v_U
\end{bmatrix} =
\vcomp{\left[\; 1^v_U \;\;  \mu^D \; \right]}{\phi}.$$ In other
words, we have a function $\phi \co U_1 \times_{D_0} D_1
\rightarrow C_1 \times_{C_0} U_1$ such that
\begin{equation} \label{equ:horizontal_mnd_mor_of_categories_0}
\phi(u,1_{tu})=(1_{su},u)
\end{equation}
for all $u \in U_1$ and
\begin{equation} \label{equ:horizontal_mnd_mor_of_categories_i}
\phi^C(\phi^U(u,d),d') \circ \phi^C(u,d)=\phi^C(u,d'\circ d)
\end{equation}
\begin{equation} \label{equ:horizontal_mnd_mor_of_categories_ii}
\phi^U(\phi^U(u,d),d')=\phi^U(u,d'\circ d).
\end{equation}
Note that if $D$ and $K$ have just one object, then equation
\eqref{equ:horizontal_mnd_mor_of_categories_ii} and the unit equation
\eqref{equ:horizontal_mnd_mor_of_categories_0} essentially say $\phi^U$ defines
a left monoid action of $D_1$ on $U_1$.  Horizontal composition of horizontal
morphisms in $\bMnd(\Sp)$ is by pullback, and the $\phi$-parts compose as in
equation \eqref{equ:phi_composition}. The horizontal identities are as in span, with $\phi$ the identity on $C_1$.

Finally, a square
\begin{equation} \label{equ:square_in_Mnd(Span)}
  \begin{array}{c}
\xymatrix{A_* \ar[r]^{(U,\phi)} \ar[d]_{(J_1,J_0)} \ar@{}[dr]|{\alpha} &
B_* \ar[d]^{(K_1,K_0)}
\\ C_* \ar[r]_{(V,\psi)} & D_*}
\end{array}
\end{equation}
in $\bMnd(\Sp)$ is a square $\alpha$ in $\Sp$ such that
$$\vcomp{\phi}{\left[\; J_1 \;\; \alpha \;\right]}=\vcomp{\left[\; \alpha \;\; K_1 \;\right]}{\psi},$$
in other words
$$(J_1(\phi^A(u,b)),\;\alpha(\phi^U(u,b)))=(\psi^C(\alpha(u),K(b)) ,\; \psi^V(\alpha(u),K(b)).$$
\end{examp}

\begin{rmk} \label{rmk:phi_as_commuting_past}
One way to think of a horizontal endomorphism map
$\phi$ is as an assignment that converts a path
$$\xymatrix{ \ar[r]^{\in U_1} & \ar[r]^{\in D_1}  & }$$ to a path
$$\xymatrix{ \ar[r]^{\in C_1} & \ar[r]^{\in U_1}  & }$$
in a way compatible with unit and composition.
\end{rmk}

Now that we understand the double categories involved,
we can give the proof of Proposition~\ref{freemonad_in_Span}.
Since the double adjunction
\eqref{equ:DiGraph_Cat_Double_Adjunction} is {\it vertical} rather
than horizontal, we use the transpose of the characterizations in
Theorem~\ref{thm:double_adjunction_def_iff_v_pseudo}. We cannot simply
transpose the double categories and double functors in
\eqref{equ:DiGraph_Cat_Double_Adjunction} in order to apply the
non-transposed Theorem~\ref{thm:double_adjunction_def_iff_v_pseudo},
because our notions of monads in a double category and their various
morphisms prefer the horizontal direction as distinguished.

\begin{proof}{Proof of Proposition~\ref{freemonad_in_Span}.}
We first describe $F$, then check the conditions of (transposed) Theorem~\ref{thm:double_adjunction_def_iff_v_pseudo}.
On objects and vertical morphisms (that is, on directed graphs and
their morphisms), $F$ is the free category functor.  On a horizontal morphism
$(U,\phi)\co G_* \rightarrow G_*'$ in $\bEnd(\Sp)$ as in
\eqref{equ:horizontal_morphism_in_End(Span)}, we have $F(U)_1:=U_1$.  The
function $\phi$ extends to $F(\phi)$ by Remark~\ref{rmk:phi_as_commuting_past}
and the fact that morphisms in the free category on a (non-reflexive) graph are
paths of edges.  On $F(U)_1 \times_{G_0'} G_1'$, the function $F(\phi)$ is
simply $\phi$.  On $F(U)_1 \times_{G_0'} F(G_*')_1$, the function $F(\phi)$ is
defined by moving the element of $U_1$ across the path, one edge at a time using
$\phi$.  For example,
\begin{equation} \label{equ:last_line_one}
\begin{array}{c}
\xymatrix@C=5pc{ \ar[r]^u & \ar[r]^g & \ar[r]^h & \\
\ar[r]^{\phi^G(u,g)} & \ar[r]^{\phi^U(u,g)} & \ar[r]^h & \\
\ar[r]^{\phi^G(u,g)} & \ar[r]^{\phi^G(\phi^U(u,g),h)} & \ar[r]^{\phi^U(\phi^U(u,g),h)} &  }
\end{array}
\end{equation}
which is the same as below.
\begin{equation} \label{equ:last_line_two}
\begin{array}{c}
\xymatrix@C=5pc{ \ar[r]^u & \ar[r]^{h\circ g} & \\
\ar[r]^{\phi^G(u, h \circ g)} & \ar[r]^{\phi^U(u,h\circ g)} &  }
\end{array}
\end{equation}
The equality of the composites in the last lines of the respective displays
\eqref{equ:last_line_one} and \eqref{equ:last_line_two} shows that $F(\phi)$
satisfies the composition rules in
\eqref{equ:horizontal_mnd_mor_of_categories_i} and
\eqref{equ:horizontal_mnd_mor_of_categories_ii} by definition.  Similarly,
\eqref{equ:horizontal_mnd_mor_of_categories_0} holds by definition and the fact
that our directed graphs are non-reflexive.  Concerning the definition of $F$ on
squares, the double functor $F$ takes a square $\alpha$ in $\bEnd(\Sp)$ as in
\eqref{equ:square_in_End(Span)} to the square $F\alpha$ in $\bMnd(\Sp)$ as in
\eqref{equ:square_in_Mnd(Span)} which has the same middle function $U_1
\rightarrow V_1$ as $\alpha$, but the left and right vertical morphisms are the
unique functors on the free categories that extend the directed graph morphisms
on the left and right of $\alpha$.  For this reason, $F$ clearly preserves
vertical composition of vertical morphisms and squares.  It also preserves
horizontal composition because the horizontal composition in both double
categories is defined via pullback.  Also the $\phi$ part of $F(V \circ U)$ is
the appropriate composite of the $\phi$-parts of $U$ and $V$ by an inductive
verification using the ``switching'' point of view on $\phi$ as just discussed.
Thus $F$ is a {\it strict} double functor.

\medskip

We use the transpose of the local description of double
adjunctions in
Theorem~\ref{thm:double_adjunction_def_iff_v_pseudo} to
prove that $F \dashv G$ is a vertical double adjunction. To simplify
our work with the transposed characterization, we introduce the
notations
$$\bMnd(\Sp)\begin{pmatrix}FU \\ V \end{pmatrix} \;\;\;\text{ and }\;\;\;
\bEnd(\Sp)\begin{pmatrix} U \\ GV \end{pmatrix}$$
to mean the set of squares in $\bMnd(\Sp)$ with vertical domain $FU$
and vertical codomain $V$, and the set of squares in $\bEnd(\Sp)$ with
vertical domain $U$ and vertical codomain $GV$. This notation is the
transpose of the notation in equation \eqref{hom_notation}. We
define a bijection
\begin{equation} \label{equ:End(Span)_Mor(Span)_bijection}
\varphi^U_V\co \bMnd(\Sp)\begin{pmatrix}FU \\ V \end{pmatrix}\xymatrix{
\ar[r] & }\bEnd(\Sp)\begin{pmatrix} U \\ GV \end{pmatrix}
\end{equation}
$$\begin{array}{c} \xymatrix@R=3pc@C=3pc{FA_* \ar[r]^{F(U,\phi)}
\ar[d]_{J} \ar@{}[dr]|{\alpha} & FB_* \ar[d]^{K}
\\ C_* \ar[r]_{(V,\psi)} & D_*} \end{array} \xymatrix{ \ar@{|->}[r] & }
\begin{array}{c}
\xymatrix@R=3pc@C=3pc{A_* \ar[r]^{(U,\phi)} \ar[d]_{J_{\text{res}}}
\ar@{}[dr]|{\alpha_{\text{res}}} & B_* \ar[d]^{K_{\text{res}}} \\ GC_*
\ar[r]_{G(V,\psi)} & GD_*} \end{array}$$ that is compatible with horizontal
composition.  The subscript $\text{res}$ means restriction: the maps
$J_{\text{res}}$ and $K_{\text{res}}$ are the restrictions of the functors $J$
and $K$ to the directed graphs $A_*$ and $B_*$, while $\alpha_{\text{res}}$ has
the same exact middle function $U_1 \rightarrow V_1$ as $\alpha$ does.  The
square $\alpha_{\text{res}}$ is restricted only in the sense that its horizontal
domain and codomain are restricted.  Since the middle function of $\alpha$ is
the same as that of $\alpha_{\text{res}}$, the function $\varphi^U_V$ is
manifestly injective.  If $\alpha'$ is a square in $\bEnd(\Sp)\begin{pmatrix} U
\\ GV \end{pmatrix}$, then we use the bijection $J \leftrightarrow
J_{\text{res}}$ to find the horizontal domain and codomain of
$(\varphi^U_V)^{-1}(\alpha')$, and define the middle function of
$(\varphi^U_V)^{-1}(\alpha')$ to be that of $\alpha'$.  This proves the
surjectivity of $\varphi^U_V$.

To see that $\varphi(\left[\; \alpha  \; \; \beta \;
\right])=\left[\; \varphi(\alpha) \; \; \varphi(\beta) \; \right]$,
we only need to observe that $(\alpha \times_{K_0} \beta)_{\text{res}}$ is
the same as $\alpha_{\text{res}} \times_{(K_\text{res})_0} \beta_{\text{res}}$ because the
diagrams, from which we are forming the pullbacks, are exactly the
same. Namely,
$$\xymatrix{(FA_*)_0 \ar[d]_{J_0} & \ar[l] \ar[d]_\alpha F(U)_1
\ar[r] & (FB_*)_0 \ar[d]^{K_0} & \ar[l] F(W)_1
\ar[r] \ar[d]^\beta & (FH_*)_1 \ar[d]^{L_0}
\\ C_0 & \ar[l] V_1 \ar[r] & D_0 & \ar[l] X_1 \ar[r] & D_0}$$
is exactly the same as
$$\xymatrix{A_0 \ar[d]_{(J_{\text{res}})_0} & \ar[l] \ar[d]_{\alpha_{\text{res}}} U_1
\ar[r] & B_0 \ar[d]^{(K_{\text{res}})_0} & \ar[l]
W_1 \ar[r] \ar[d]^{\beta_{\text{res}}} & H_1 \ar[d]^{(L_{\text{res}})_0}
\\ (GC_*)_0 & \ar[l] (GV)_1 \ar[r] & (GD_*)_0 & \ar[l] (GX)_1 \ar[r] & (GD_*)_0.}$$

It only remains to check the naturality of $\varphi^U_V$ in $U$ and
$V$, but that is similar to the naturality of the ordinary free
category functor-forgetful functor adjunction, the only difference
is that here we use {\it vertical} pre- and post-composition of {\it
squares}.

In summary, the bijection $\phi^U_V$ in
\eqref{equ:End(Span)_Mor(Span)_bijection} is compatible with
horizontal composition and natural in the horizontal morphisms $U$
and $V$, so $F$ is vertical double left adjoint to $G$ by the transpose of
Theorem~\ref{thm:double_adjunction_def_iff_v_pseudo}.
\end{proof}

In the next section we analyze the free-monad adjunction in
a more general setting.  In Section~\ref{sec:Characterization_of_Existence_of_Eilenberg-Moore_Objects}
we study another important example of double adjunction, namely an
Eilenberg--Moore type adjunction.


\section{Free Monads in Double Categories with Cofolding}
\label{sec:Construction_of_Free_Monads}

In this section we remove several hypotheses from our main theorem
in~\cite{FioreGambinoKock:DoubleMonadsI} and strengthen its conclusion to obtain
Theorem~\ref{thm:existence_of_free_monads},
which says that if a double category $\bbD$ with cofolding admits the
construction of free monads in its horizontal 2-category, then $\bbD$ admits the
construction of free monads as a double category.
Since the free--forgetful double adjunction is a {\em vertical} adjunction,
it is remarkable that it can be inferred from the free--forgetful adjunction
in the {\em horizontal} $2$-category.
We first recall free monads
on endomorphisms in a 2-category in
Definition~\ref{def:admits_construction_of_free_monads:2-category_version},
which is due to Staton~\cite[Theorem~6.1.5]{StatonS:namppc} in the case
$\mathbf{K}=\mathbf{Cat}$, and is treated in general in our previous paper
\cite[Theorem~1.1]{FioreGambinoKock:DoubleMonadsI}.

\begin{defn} \label{def:admits_construction_of_free_monads:2-category_version}
Let $\mathbf{K}$ be a 2-category. We say $\mathbf{K}$ {\it admits
the construction of free monads} if either of the two following equivalent
conditions hold.
\begin{enumerate}
\item
\label{def:admits_construction_of_free_monads:2-category_version:universal_counit_component}
For every endomorphism $(Y,Q)$ there exists a monad $(Y,Q^\text{\rm free})$ and a
2-cell $\iota \co Q \rightarrow Q^\text{\rm free}$ in $\mathbf{K}$ such that the
endomorphism map $(1_Y,\iota_Q)\co (Y,Q^\text{\rm free}) \rightarrow (Y,Q)$ is
universal in the sense that for every monad $(X,P)$,
post-composition with $(1_Y,\iota_Q)$ induces an isomorphism of
categories
$$
\xymatrix@C=5pc{\Mnd_{\mathbf{K}} ((X,P),(Y,Q^\text{\rm free})) \ar[r]^-{(1_Y,\iota_Q) \circ U(-)} & \End_{\mathbf{K}}(U(X,P),(Y,Q)),}
$$
where $U\co \Mnd(\mathbf{K}) \rightarrow \End(\mathbf{K})$ is the
forgetful 2-functor.
\item
The forgetful functor $U\co \Mnd(\mathbf{K}) \rightarrow
\End(\mathbf{K})$ admits a {\it right} 2-adjoint $R\co
\End(\mathbf{K}) \rightarrow \Mnd(\mathbf{K})$ with a counit
$\varepsilon$ such that the underlying morphism in $\mathbf{K}$ of
each counit component $\varepsilon_{(Y,Q)} \co UR(Y,Q) \rightarrow (Y,Q)$
is $1_Y$.
\end{enumerate}
\end{defn}


\begin{rmk}
The reason Definition~\ref{def:admits_construction_of_free_monads:2-category_version}
requires a {\it right} adjoint to the forgetful functor (as opposed
to an expected {\it left} adjoint) is the choice of the direction of
2-cell in the definition of endomorphism map and monad map, as we
now explain.  Briefly, this right adjoint restricts to a left adjoint when we
consider monads and endomorphisms on a fixed object $Y$.
In detail,
consider a fixed object $Y$ of the 2-category $\mathbf{K}$. The
{\it category of endomorphisms on $Y$}, denoted $\End(Y)$, has
objects endomorphisms on $Y$.  The morphisms in $\End(Y)$ are
endomorphism maps with underlying morphism the identity on $Y$, that
is, endomorphism maps of the form $(1_Y,\phi)\co (Y,Q_1)
\rightarrow (Y,Q_2)$. We follow the convention of Street~\cite{Street:formal-monads}
for the 2-cell $\phi$, namely $\phi\co Q_2 1_Y \rightarrow 1_Y Q_1$.
There are no compatibility requirements on
$\phi$. The {\it category of monads on $Y$}, denoted $\Mnd(Y)$, has
objects monads on $Y$. The morphisms in $\Mnd(Y)$ are monad maps
with underlying morphism the identity on $Y$, that is, morphisms are
monad maps of the form $(1_Y,\psi)\co (Y,M_1) \rightarrow
(Y,M_2)$.  Again, we follow Street's convention in
\cite{Street:formal-monads} for the 2-cell $\psi$, namely $\psi:M_2
1_Y \rightarrow 1_Y M_1$. The 2-cell $\psi$ is required to be
compatible with the unit and multiplication of the monads $M_1$ and
$M_2$.

The variance in
Definition~\ref{def:admits_construction_of_free_monads:2-category_version}
restricts to the expected one for monads on the fixed object $Y$,
that is, the 2-category $\mathbf{K}$ {\it is said to admit the
construction of free monads on $Y$} if the forgetful functor $U_Y
\co \Mnd(Y) \rightarrow \End(Y)$ admits a {\it left} adjoint. If
$\mathbf{K}$ admits the construction of free monads in the sense of
Definition~\ref{def:admits_construction_of_free_monads:2-category_version},
then $\mathbf{K}$ admits the construction of free monads on each
object $Y$.
\end{rmk}

\begin{rmk}
In
Definition~\ref{def:admits_construction_of_free_monads:2-category_version}~\ref{def:admits_construction_of_free_monads:2-category_version:universal_counit_component},
the isomorphism of categories commutes with the evident forgetful functors
\begin{equation*} \label{equ:free_monads_in_a_2-category:diagram_over_K}
\xymatrix{ \Mnd(\mathbf{K})( (X,P), (Y, Q^\text{\rm free}) ) \ar[rr]^\cong
\ar[dr] & & \End(\mathbf{K})( U(X,P), (Y,Q)) \ar[dl] \\
 & \mathbf{K}(X,Y), & }
\end{equation*}
since the underlying morphisms and 2-cells in $\mathbf{K}$ are composed with
(whiskered with) $1_Y$.
\end{rmk}

%

The following definition is slightly different from
\cite[Definition~2.8]{FioreGambinoKock:DoubleMonadsI} in that it
insists on the vertical triviality of the unit.
\begin{defn} \label{def:a_double_category_admits_the_construction_of_free_monads}
A double category $\bbD$ is said to {\it admit the construction of free monads}
if the forgetful double functor $U \co \bMnd(\bbD) \rightarrow \bEnd(\bbD)$
admits a vertical {\em left} double adjoint $R$ with a unit $\eta$ such that the
underlying vertical morphism in $\bbD$ of each unit component $\eta_{(Y,Q)} \co
(Y,Q) \rightarrow UR(Y,Q)$ is $1_Y^v$.
\end{defn}

We shall shortly prove that if $\bbD$ has a cofolding, then the existence of
free monads in $\bfH\bbD$ implies the existence of free monads in $\bbD$.  This
amounts to extending an adjunction from the horizontal $2$-categories to a {\em
vertical} double adjunction.  We first extend the $2$-adjunction of horizontal
$2$-categories to a horizontal double adjunction.
For both results, observe that the double-categorical notions of endomorphism,
monad, and the forgetful double functor $U \co \bMnd(\bbD) \to \bEnd(\bbD)$ are
essentially notions of the horizontal $2$-category.  More precisely we can
identify $\bfH U \co \bfH \bMnd(\bbD) \to \bfH \bEnd(\bbD)$ with the forgetful
$2$-functor $\Mnd(\bfH\bbD) \to \End(\bfH \bbD)$.

\begin{prop}\label{prop:free_to_horizontal_adj}
  Let $\bbD$ be a double category with cofolding $\Lambda$.  Suppose that the
  horizontal 2-category $\bfH\bbD$ admits the construction of free
  monads in the sense of
  Definition~\ref{def:admits_construction_of_free_monads:2-category_version}.
  Then the $2$-adjunction
\begin{equation*}
\xymatrix@C=4pc{ \Mnd(\bfH\bbD) \ar@/^1pc/[r]^{U} \ar@{}[r]|\perp & \ar@/^1pc/[l]^{R} \End(\bfH\bbD)}
\end{equation*}
extends to a horizontal double adjunction
\begin{equation*}
\xymatrix@C=4pc{\bMnd(\bbD) \ar@/^1pc/[r]^{U} \ar@{}[r]|\perp &
\ar@/^1pc/[l]^{R} \bEnd(\bbD) .}
\end{equation*}
\end{prop}

\begin{proof}
  By the above remark, $U$ automatically extends to a double functor.  The main
  point is to extend $R$, which relies on the cofoldings on $\bEnd(\bbD)$ and
  $\bMnd(\bbD)$ guaranteed by
  Proposition~\ref{prop:cofolding_on_D_induces_cofolding_on_Mnd(D)_and_on_End(D)},
  and the crucial fact that the counit of the $2$-adjunction $U \dashv R$ has
  components of the form $\varepsilon_{(Y,Q)}=(1_Y^h,\iota_Q)$.  The 2-functor
  $R$ is defined on (horizontal) endomorphism maps $(F,\phi)\co (X,P)
  \rightarrow (Y,Q)$ and endomorphism 2-cells $\alpha \co (F_1,\phi_1)
  \Rightarrow (F_2,\phi_2)$ by the equations
\begin{equation} \label{equ:thm:horizontal_construction_of_monads_implies_all_others:morphisms}
\left[\; UR(F,\phi) \;\; (1_Y^h,\iota_Q) \right] = \left[ \; (1_X^h,\iota_P) \;\; (F,\phi) \; \right]
\end{equation}
\begin{equation} \label{equ:thm:horizontal_construction_of_monads_implies_all_others:2-cells}
\left[\; UR\alpha \;\; i^v_{(1_Y^h,\iota_Q)} \right] = \left[ \; i^v_{(1_X^h,\iota_P)} \;\; \alpha \; \right].
\end{equation}
If $(u,\overline{u})$ is a vertical endomorphism map, then
$R(u^*,\Lambda(\overline{u}))=:(Ru^*,R\Lambda(\overline{u}))$ is defined by
\eqref{equ:thm:horizontal_construction_of_monads_implies_all_others:morphisms}.
We see from
\eqref{equ:thm:horizontal_construction_of_monads_implies_all_others:morphisms}
that the underlying horizontal morphism of $Ru^*$ is $u^*$, so by
Proposition~\ref{prop:bijection_for_fixed_underlying_morphism} we may apply
$\Lambda^{-1}$ to $R\Lambda(\overline{u})$ to obtain
$R(u,\overline{u}):=(u,\Lambda^{-1}R\Lambda(\overline{u}))$ with underlying
vertical morphism $u$.  A similar argument using
equation~\eqref{equ:thm:horizontal_construction_of_monads_implies_all_others:2-cells}
defines $R$ on squares of $\bEnd(\bbD)$.
By construction, the double functors $R$ and $U$ are compatible with the
cofoldings, so the 2-adjunction $\bfH U \dashv \bfH R$ extends to a horizontal
double adjunction by
Proposition~\ref{prop:folding_double_adjunction_iff_horizontal_double_adjunction}.
\end{proof}

\begin{thm}[Reduction of construction of free monads to horizontal 2-category]
  \label{thm:existence_of_free_monads}
Let $\bbD$ be a double category with cofolding.  If the horizontal 2-category
$\bfH\bbD$ admits the construction of free monads in the sense of
Definition~\ref{def:admits_construction_of_free_monads:2-category_version}, then
the double category $\bbD$ admits the construction of free monads in the sense
of
Definition~\ref{def:a_double_category_admits_the_construction_of_free_monads}.
\end{thm}

\begin{proof}
  By Proposition~\ref{prop:free_to_horizontal_adj} the $2$-functor $R$ of
  Definition~\ref{def:admits_construction_of_free_monads:2-category_version}
  extends
  to a double functor $R\co\bEnd(\bbD)\to \bMnd(\bbD)$.  We shall check that $R$ is
  vertical left double adjoint to $U \co \bMnd(\bbD) \rightarrow \bEnd(\bbD)$
  using the transpose of
  Theorem~\ref{thm:double_adjunction_descriptions}~\ref{Guniversalarrow}, which
  requires functors
$$\xymatrix{R_0\co
(\Obj \bEnd(\bbD), \Hor \bEnd(\bbD)) \ar[r] & (\Obj \bMnd(\bbD),
\Hor \bMnd(\bbD))}$$
$$\xymatrix{\eta \co (\Obj \bEnd(\bbD), \Hor \bEnd(\bbD)) \ar[r] &
(\Ver \bEnd(\bbD), \Sq \bEnd(\bbD))}$$
such that for each horizontal morphism $(F,\phi)$ in $\bEnd(\bbD)$
the square $\eta_{(F,\phi)}$ is of the form $$\xymatrix@C=4pc{(X,P)
\ar[r]^-{(F,\phi)} \ar[d]_-{\eta_{(X,P)}}
\ar@{}[dr]|-{\eta_{(F,\phi)}}  &  (Y,Q)  \ar[d]^-{\eta_{(Y,Q)}} \\
UR_0(X,P) \ar[r]_{UR_0(F,\phi)} & UR_0(Y,Q)}$$ and is universal from
$(F,\phi)$ to $U$.

We define $R_0$ as the horizontal 1-adjoint already present, namely
$R_0(X,P):=(X,P^\text{\rm free})$ and $R_0(F,\phi)\co (X,P^\text{\rm free})
\rightarrow (Y,Q^\text{\rm free})$ is the unique (horizontal) monad morphism
such that $(1_Y^h,\iota_Q) \circ UR_0(F,\phi) = (F,\phi) \circ (1_X^h\iota_P)$.

The functor $\eta$ on objects is $\eta_{(X,P)}:=(1^v_{X},
\left(\Lambda^{\bbD}\right)^{-1}(\iota_P))=(1^v_{X}, \iota_P)$.  Here
$\Lambda^{\bbD}$ is the cofolding on $\bbD$, and we are using
Proposition~\ref{prop:cofolding_on_D_induces_cofolding_on_Mnd(D)_and_on_End(D)}
for the cofolding on $\bEnd(\bbD)$, the bijection in
Proposition~\ref{prop:bijection_for_fixed_underlying_morphism} for the fixed
vertical morphism $(1^v_X)$, and the fact that $(1^v_X)^*=1^h_X$.  For a
horizontal endomorphism map $(F,\phi)$, we define $\eta_{(F,\phi)}$ to be
$\left(\Lambda^{\bEnd(\bbD)}\right)^{-1}$ of the vertical identity square
\[
\xymatrix@C=4pc{UL_0(X,P) \ar[r]^{(1_X^h,\iota_P)}  \ar@{=}[d] & (X,P)
\ar[r]^{(F,\phi)}  \ar@{}[d]|{i^v} & (Y,Q) \ar@{=}[d]
\\ UL_0(X,P)  \ar[r]_-{UL_0(F,\phi)} & UL_0(Y,Q) \ar[r]_-{(1_Y^h,\iota_Q)} & (Y,Q)}
\]
in $\bEnd(\bbD)$.

For the universality of $\eta_{(Y,Q)}$ concerning vertical morphisms, we must
prove for each endomorphism $(Y,Q)$ and each monad $(X,P)$ that
$$\xymatrix@C=5pc{\text{Ver}_{\bMnd(\bbD)}(Y,Q^\text{\rm free}),(X,P))
\ar[r]^-{U(-) \circ (1^v_Y,\iota_Q)} & \text{Ver}_{\bEnd(\bbD)}((Y,Q),U(X,P)) }$$
is a bijection. For injectivity, if
$U(u,\overline{u})\circ(1^v_Y,\iota_Q)=U(v,\overline{v})\circ(1^v_Y,\iota_Q)$,
then $u=v$, and the coholonomy on $\bEnd(\bbD)$ gives us
$$(1^h_Y,\iota_Q) \circ U(u^*,\Lambda(\overline{u}))=(1^h_Y,\iota_Q)
\circ U(v^*,\Lambda(\overline{v})),$$ so $\Lambda(\overline{u})=\Lambda(\overline{v})$
by horizontal universality of $(1^h_Y,\iota_Q)$. Finally, $\overline{u}=\overline{v}$
by Proposition~\ref{prop:bijection_for_fixed_underlying_morphism}. For surjectivity,
if $(w, \overline{w})\co(Y,Q) \rightarrow U(X,P)$ is a vertical endomorphism map,
the horizontal universality of $(1^h_Y,\iota_Q)$ guarantees a horizontal monad map
$(F,\phi)\co (X,P) \rightarrow (Y,Q^\text{\rm free})$ such that
$(1^h_Y,\iota_Q) \circ U(F,\phi)=(w^*,\Lambda(\overline{w}))$.
Then $F=w^*$, and we may take $(u,\overline{u})=(w,\Lambda^{-1}( \left[\; \phi
\;\;  \iota_Q \;\right])$ so that $U(u,\overline{u})\circ(1^v_Y,\iota_Q)=(w,\overline{w})$,
again by Proposition~\ref{prop:bijection_for_fixed_underlying_morphism}.

We next prove that the square $\eta_{(F,\phi)}$ is vertically
universal, that is, the map
\begin{equation} \label{equ:vertical_universality_in_proof_free_horizontal_implies_free_vertical}
\bMnd(\bbD)\begin{pmatrix}R_0(F,\phi) \\ (F',\phi') \end{pmatrix}
\xymatrix{ \ar[r] & } \bEnd(\bbD)\begin{pmatrix} (F,\phi) \\
U(F',\phi') \end{pmatrix}
\end{equation}
\[
\beta \xymatrix{ \ar@{|->}[r] &  } \begin{bmatrix} \eta_{(F,\phi)}
\\ U\beta \end{bmatrix}.
\]
is a bijection (recall Definition~\ref{def:horizontally_universal_square}.).
The notation $\bMnd(\bbD)\begin{pmatrix}R_0(F,\phi) \\ (F',\phi') \end{pmatrix}$
indicates the set of monad squares with top horizontal arrow $R_0(F,\phi)$ and
bottom horizontal arrow $(F',\phi')$.  The notation $\bEnd(\bbD)\begin{pmatrix}
(F,\phi) \\
U(F',\phi') \end{pmatrix}$ indicates the set of endomorphism squares
with top horizontal arrow $(F,\phi)$ and bottom horizontal arrow $U(F',\phi')$.

Since we have already checked the universality of $\eta_{(Y,Q)}$ with respect to
vertical morphisms, and since squares with distinct vertical arrows are
distinct, it suffices to prove a bijection for monad squares which additionally
have the left and right vertical arrows fixed, so we consider monad squares of
the form
\[
\xymatrix{(X,P) \ar[r]^{R_0(F,\phi)} \ar[d]_{(u,\bar{u})}
\ar@{}[dr]|\beta & (Y,Q^\text{\rm free}) \ar[d]^{(v,\bar{v})} \\ (X',P')
\ar[r]_{(F',\phi')} & (Y',Q'). }
\]
We factor the map in
\eqref{equ:vertical_universality_in_proof_free_horizontal_implies_free_vertical}
(for fixed $(u,\bar{u})$ and $(v,\bar{v})$), into a sequence of
bijections.
\[
\aligned
\beta & \leftrightarrow \Lambda^{\bMnd(\bbD)}(\beta) \\
& \leftrightarrow \left[U\Lambda^{\bMnd(\bbD)}(\beta) \;\; i^v_{(1_Y,\iota_Q)} \right] \\
& \leftrightarrow \begin{array}{c} \begin{bmatrix} i^v_{(u,\bar{u})^*} & i^v \\
U\Lambda^{\bMnd(\bbD)}(\beta) & i^v_{(1_Y,\iota_Q)} \end{bmatrix} \end{array} \\
& \leftrightarrow  \begin{array}{c} \begin{bmatrix} \eta_{(F,\phi)}
\\ U\beta  \end{bmatrix} \end{array}
\endaligned
\]
The last bijection is $(\Lambda^{\bEnd(\bbD)})^{-1}$ and relies on the fact that
$U$ is compatible with the cofoldings $\Lambda^{\bMnd(\bbD)}$ and
$\Lambda^{\bEnd(\bbD)}$.
\end{proof}

\begin{rmk}
  Note that the conclusion of Theorem~\ref{thm:existence_of_free_monads}, that
  $\bbD$ admits the construction of free monads, amounts to a vertical
  double adjunction, the free-monad double functor $R$ being the {\em left} double
  adjoint.  Since $\bfV: \mathbf{DblCat_v} \to \mathbf{Cat}$ is a $2$-functor,
  we obtain (in the situation of the Theorem) also a 2-adjunction
\begin{equation*}
\xymatrix@C=4pc{\bfV \bEnd(\bbD) \ar@/^1pc/[r]^{\bfV R} \ar@{}[r]|\perp & \ar@/^1pc/[l]^{\bfV U} \bfV \bMnd(\bbD).}
\end{equation*}
\end{rmk}

\section{Existence of Eilenberg--Moore Objects}
\label{sec:Characterization_of_Existence_of_Eilenberg-Moore_Objects}

   The double functor $\bMnd(\bbD) \to \bbD$ which to a monad
   associates its underlying object, has a horizontal double
   right adjoint $\Inc_\bbD$ which to an object in $\bbD$ associates
   the trivial monad on it:
   \begin{equation*}
   \xymatrix@C=4pc{\bMnd(\bbD) \ar@/^1pc/[r]^{\Und} \ar@{}[r]|{\perp} &
   \ar@/^1pc/[l]^{\Inc_\bbD} \bbD.}
   \end{equation*}
   In this final section we analyze when $\Inc_\bbD$ has a further
   right double adjoint.

In Street's article~\cite{Street:formal-monads}, a 2-category $\bfK$
is said to {\it admit the construction of algebras} if the inclusion
2-functor $\Inc_{\bfK} \co \bfK \rightarrow \Mnd(\bfK)$ admits a right
2-adjoint $\text{Alg}_\bfK\co\Mnd(\bfK) \rightarrow  \bfK$. Synonymously,
we say $\bfK$ {\it admits Eilenberg--Moore objects}. For a monad
$(X,S)$ in $\bfK$, the object $\text{Alg}_\bfK(X,S)$ is denoted
$X^S$. A right 2-adjoint $\text{Alg}_\bfK$ exists if and only if for
each monad $(X,S)$, the presheaf
$\Mnd_{\bfK}\left(\Inc_\bfK-,(X,S)\right)$ is representable. The
representing object is then $X^S$.

The situation for monads in a double category $\bbD$ is more subtle,
as representability of the individual presheaves
$\Mnd_{\bbD}\left(\Inc_\bbD(-),(X,S)\right)$ does not suffice, and
we must consider parameterized presheaves.

\begin{defn}
Let $\bbD$ be a double category and let $\Inc_\bbD\co \bbD \rightarrow
\bMnd(\bbD)$, $I \mapsto (I,\id_I)$ be the inclusion double functor.
We say that the double category $\bbD$ {\it admits Eilenberg--Moore
objects} if $\Inc_\bbD$ admits a horizontal right double adjoint.
\end{defn}

\begin{rmk}
To an object $I$ and a monad $(X,S)$ in $\bbD$, we may associate the set
$S\Alg_I$ of $S$-algebra structures on $I$, which is the set of horizontal
monad morphisms from $(I,\id_I)$ to $(X,S)$.  This assignment extends to a
parameterized presheaf on $\bbD$ in the sense of
Definition~\ref{def:parameterized_presheaf}, namely
\begin{equation} \label{equ:S-Alg_I_as_parameterized_presheaf}
\xymatrix{\bMnd(\bbD)(\Inc_\bbD-,-)\co \bbD^{\text{\rm horop}}
\times \bbV_1\bMnd(\bbD) \ar[r] & \Sp^t}.
\end{equation}
Recall that $\bbV_1\bMnd(\bbD)$ is the double category which has the
same vertical 1-category as $\bMnd(\bbD)$, but everything else is
trivial, as in Section~\ref{sec:Notation}.
\end{rmk}


\begin{thm}[Characterization of existence of Eilenberg--Moore objects]
  \label{thm:EilenbergMooreExistence}
The inclusion double functor
$$\xymatrix{\Inc_\bbD \co \bbD \ar[r] & \bMnd(\bbD)}$$
$$\xymatrix{I \ar@{|->}[r] & (I,\id)}$$
admits a horizontal right double adjoint if and only if the
parameterized presheaf
$$\xymatrix{\Alg_-\co \bbD^{\text{\rm horop}} \times \bbV_1\bMnd(\bbD)
\ar[r] & \Sp^t}$$
is (horizontally) representable in the sense of
Definition~\ref{defn:parametrized_presheaf_representability}.
\end{thm}
\begin{proof}
By Theorem~\ref{thm:representability_characterization_of_dbladjs},
the double functor $\Inc_\bbD$ admits a horizontal right double
adjoint if and only if the parameterized presheaf
\eqref{equ:S-Alg_I_as_parameterized_presheaf} is representable, but
$\Alg_-$ is \eqref{equ:S-Alg_I_as_parameterized_presheaf} by
definition.
\end{proof}

\begin{examp}
Suppose $\bfK$ is a 2-category which admits Eilenberg--Moore objects
in the sense of 2-category theory, that is, the 2-functor $\Inc_\bfK
\co \bfK \rightarrow \Mnd(\bfK)$ admits a right 2-adjoint. Then the
double category $\overline{\bbQ} \bfK$ admits Eilenberg--Moore
objects since $\overline{\bbQ} \bfK$ and $\bMnd(\overline{\bbQ}
\bfK)=\Qbar\Mnd(\bfK)$ both have cofoldings with fully faithful coholonomies,
$\Inc_{\overline{\bbQ} \bfK}$ preserves them, and $\bfH
\Inc_{\overline{\bbQ} \bfK}=\Inc_\bfK$ admits a right 2-adjoint. See
Example~\ref{exa:twistedquintets},
Proposition~\ref{prop:cofolding_on_D_induces_cofolding_on_Mnd(D)_and_on_End(D)},
and
Corollary~\ref{cor:fully_faithful_coholonomy_implies_all_adjunctions_equivalent}.
The representing functor $G\co \bbV_1\bMnd(\overline{\bbQ} \bfK)
\rightarrow \Sp^t$ for $\Alg_-$ is the transposed opposite of the
right adjoint to $\Inc_{\bfK}$.
\end{examp}


\section*{Acknowledgements}
Most of the results of this paper were obtained in 2008
during the three authors' participation in the special
year on Homotopy Theory and Higher Categories at the Centre
de Recerca Matem\`atica in Barcelona, see the preprint \cite{FioreGambinoKock:DoubleMonadsII_CRM_Preprint}. We gratefully acknowledge the support and hospitality of the CRM.

Thomas M.~Fiore was supported at the University of
Chicago by NSF Grant DMS-0501208. At the Universitat Aut\`{o}noma de
Barcelona he was supported by grant SB2006-0085 of the Spanish
Ministerio de Educaci\'{o}n y Ciencia under the Programa Nacional de
ayudas para la movilidad de profesores de universidad e
investigadores espa\~noles y extranjeros. Thomas M.
Fiore was also supported by the Max Planck Institut f\"ur
Mathematik, and he thanks MPIM for its kind hospitality during Summer 2010
and Summer 2011.

Nicola Gambino would like to acknowledge also the hospitality
of the Institute for Advanced Study. This material is based upon work
supported by the National Science Foundation under agreement No.
DMS-0635607. Any opinions, findings and conclusions or recommendations
expressed in this material are those of the authors and do not
necessarily reflect the views of the National Science Foundation.

Joachim Kock was partially supported by grants MTM2006-11391 and
MTM2007-63277 of Spain and SGR2005-00606 of Catalonia.


\end{document}